\documentclass[a4paper,reqno,12pt]{amsart}
\usepackage{amsfonts,amssymb,amsbsy,amsmath,amsthm,mathrsfs,enumerate,verbatim}
\usepackage{pdfsync}
\usepackage{grffile}
\usepackage{cite}
\usepackage{color}
 \usepackage{graphicx,color}
\usepackage{algorithmic,algorithm}
\usepackage{breqn}
\usepackage{colortbl}
\usepackage{hhline}
\usepackage{amsthm}
\usepackage{graphicx}   % For including images
\usepackage{tikz,graphics,color,fullpage,float,epsf,caption,subcaption}
\usepackage{tikz}
\usetikzlibrary{matrix}
\usepackage{graphicx}
\usepackage{mwe}% just for the example content

\newtheorem{Theorem}{Theorem}[section]
\newtheorem{theorem}{Theorem}[section]
\newtheorem{Lemma}[Theorem]{Lemma}

\newtheorem{corollary}{Corollary}[theorem]
\newtheorem{Remark}[Theorem]{Remark}

\newtheorem{assumption}{Assumption}

\usepackage{graphics} %% add this and next lines if pictures should be in esp format
\usepackage{epsfig} %For pictures: screened artwork should be set up with an 85 or 100 line screen
\usepackage{graphicx}
\usepackage{algorithmic,algorithm}
\usepackage{epstopdf}%This is to transfer .eps figure to .pdf figure; please compile your paper using PDFLeTex or PDFTeXify.
\usepackage{amsfonts}
\usepackage{amsmath}
\usepackage{amssymb}
\usepackage{graphicx}
\usepackage{epsf}
\usepackage{mathptmx}       % selects Times Roman as basic font
\usepackage{helvet}         % selects Helvetica as sans-serif font
\usepackage{courier}        % selects Courier as typewriter font
\usepackage{type1cm}        % activate if the above 3 fonts are
\usepackage{graphicx}        % standard LaTeX graphics tool
\usepackage{multicol}        % used for the two-column index
\usepackage{amsfonts}
\usepackage{amsmath}
\usepackage{amssymb}
\usepackage{amsfonts}
\usepackage{graphicx}
\usepackage{epsf}
\usepackage{epsfig}
\usepackage[capitalize]{cleveref}

\usepackage{color}
\usepackage{afterpage}
\usepackage{verbatim}
\usepackage{algorithmic,algorithm}

\usepackage{xcolor}
\makeindex             % used for the subject index
\begin{document}

  \textheight=9 true in
   \textwidth=6.0 true in
    \topmargin 30pt
     \setcounter{page}{1}

%%%%%%%%%%%%%%%%%%%%%%%%%%%%%%%%%%%%%%%%%%%%%%%%%%%%%%%%% ALL FIGURES %%%%%%%%%%%%%%%%%%%%%%%%%%%%%%%

% choose options for [] as required from the list
% in the Reference Guide

          % not available on your system
%

%\usepackage{showkeys}

%\usepackage{refcheck}
% see the list of further useful packages
% in the Reference Guide

                       % please use the style svind.ist with
                       % your makeindex program
%===========================

%=================================================================
% Full title of the paper (Capitalized)
\title{Reconstruction of source function in a parabolic  equation \\ \vspace{.05in} using partial boundary measurements}
 
\author{T. Sharma\textsuperscript{1}}    
\address{\textsuperscript{1} Department of Mathematical Sciences, Indian Institute of Space Science and Technology, 	Thiruvananthapuram, Kerala, India.} \email{tarunsharma80065@gmail.com} 
 \author{L. Beilina \textsuperscript{2}}
 \address{\textsuperscript{2} Department of Mathematical Sciences, Chalmers University of Technology and University of Gothenburg, SE-42196 Gothenburg, Sweden.} \email{larisa@chalmers.se} 
 \author{K. Sakthivel \textsuperscript{3}}
 \address{\textsuperscript{3} Department of Mathematical Sciences, Indian Institute of Space Science and Technology, 	Thiruvananthapuram, Kerala, India.} \email{sakthivel@iist.ac.in  }
\date{}

\begin{abstract}
In this paper, we present the analytical and numerical study of the optimization
approach for determining the space-dependent source function
in the parabolic inverse source problem using partial boundary
measurements. The Lagrangian approach for the solution of the optimization problem is presented,
and optimality conditions are derived. The proof of the Fr\'echet differentiability of the
regularized Tikhonov functional and the existence result for the solution of the inverse source problem are established.  A  local stability estimate for the unknown source term is also presented.
The numerical examples justify the theoretical investigations using the conjugate gradient method (CGM) in 2D and 3D tests with noisy data.
\end{abstract}

\maketitle
\graphicspath{
{/FIGURES/}
{/chalmers/groups/larisa_2/PhD_students/TarunSharma/MATLAB08_01_2025/SourcewithSensitivity/}
  {pics/}}
\section{Introduction}
The inverse problems in PDEs represent an intriguing area where mathematics intersects with practical challenges. These problems focus on determining unknowns in a system of differential equations using observable data. The boundary measurement problems involve determining unknown parameters or functions within a domain based on observed data from the boundary. These problems are often ill-posed, necessitating the use of regularization techniques to obtain stable solutions. In this paper, we present a theoretical analysis and develop a reconstruction algorithm to determine the spatially distributed source function in a parabolic PDE from a partial boundary measurement using an optimization technique. 

Let $ \Omega \subset \mathbb R^N, N=2,3$ be a convex bounded domain with a sufficiently smooth boundary $\partial \Omega \in C^2,$ and $(0,T)$ denotes the time interval with the final time $T$.
Let the boundary $\partial \Omega$ be such that $\partial \Omega =\partial _1 \Omega \cup \partial _2 \Omega $. Here, $\partial _1 \Omega$ is the top part of $\Omega$, 
and $\partial _2 \Omega$  denotes other sides of the domain $\Omega$.
Let us denote $ \Omega_T := \Omega \times (0,T), \partial \Omega_T := \partial \Omega \times (0,T), \partial_1 \Omega_T := \partial_1 \Omega \times (0,T)$.\par
Consider an inverse problem of recovering the unknown source function $F(x)$ in a general parabolic equation.
\begin{equation}\label{prb1}
\begin{cases}
a(x) \frac{\partial u(x,t)}{\partial t}  - \triangle u(x,t)  &= F(x)G(x,t), 
~~ (x,t) \in \Omega_T,\\
  u(x,0) &= 0,~~ x \in \Omega,\\
\partial_\nu u(x,t) &= 0, ~~ (x,t)\in \partial\Omega\times(0,T).    
\end{cases}  
\end{equation}
from the partial boundary  data  $\widetilde u(x,t)$ measured at the  boundary $\partial_1\Omega$ in time $(0,T)$:  
\begin{equation}\label{prb2}
u(x,t) = \widetilde u(x,t), \qquad (x,t)\in\partial_1\Omega\times(0,T). 
\end{equation}
Here, $F(x)$ is space-dependent and $G(x,t)$ is space and time-dependent source functions,
respectively, $\partial_\nu(\cdot)$ denotes the normal derivative of
$\partial \Omega$, where $\nu$ is the outward unit normal vector on
the boundary $\partial \Omega$.  

The problem \eqref{prb1}-\eqref{prb2} has a lot of applications, among other
things, in medical imaging \cite{BK}. More precisely, the heat source identification problems are the most commonly encountered inverse problems in heat conduction or diffusion. In many branches of science and engineering, e.g., crack identification, geophysical prospecting, pollutant detection, and designing the final state in melting and freezing processes, the characteristics of sources are often unknown and need to be determined. These problems have been studied over several decades due to their significance in a variety of scientific and engineering applications (see,\cite{cannon1998structural,savateev1995problems,Reeve,Isakov1991, hasanouglu2021introduction}  and
 references therein). In the modeling of air pollution phenomena,
 the source function in \eqref{prb1} is considered as the source pollutant. Thus, an accurate estimation of the
pollutant source is crucial to environmental safeguards in cities with high populations (see, \cite{Ebel, Cheng, Isakov} and
 references therein). 
 
 The inverse source identification problem, which is a classic ill-posed problem in the sense of Hadamard\cite{hadamard1964theorie}, i.e., a small perturbation in the input data may cause a dramatically large error in the solution (if it exists). In practice, it is challenging to exclude noise from measurement errors in the measured output data $\tilde u(x,t)$; as a result, an exact equality in \eqref{prb2} may not be satisfied. Therefore, apart from the other techniques for inverse problems for PDEs (see, \cite{bugheim1981global}), it should be noted that the optimization methods based on weak solution theory for PDEs play a crucial role in proving the results for inverse problems. The notion of determining the unknown parameters by optimization method is indeed classical (see, for instance,\cite{hasanouglu2021introduction,T,ivanov2013theory}) which yields a general solution for inverse problems without any uniqueness result of the solution. More precisely, this technique involves the restriction of the solution for the inverse problem under consideration into an admissible set and then finding a minimizer of a cost functional (with a suitable stabilizer or regularizer) as the general solution for the desired problem. In works \cite{Hasanov2007, Hasanov2011}, via a weak solution approach combined with an optimization method, the author elaborated on the method for simultaneous determination of source terms in a linear parabolic problem from the final time measurement data and source identification from the single Dirichlet-type measured output data, respectively. Further, this method has been effectively used for the inverse problems of higher order PDEs (\cite{anjuna2021determination}) and system of PDEs (\cite{Sak,NKM}). 
 
There are so many articles on the different methods for solving inverse source problems. Let us quickly review some of the numerical research works performed in this direction. In \cite{Johansson}, the source function $f(x)$ is reconstructed for the heat equation by the iterative BEM regularizing algorithm. For the linear parabolic problem, the author in \cite{zhang2012} used the variational approach to find the space-time dependent source functions with the final time measurement. The paper \cite{Erdem} studied the source identification for the linear heat equation by employing the finite difference technique with the CGM method. The source ($f(x)$) identification problem with the final time measured data for the linear parabolic equation, where the source term is given by $F(x,t) = f(x)g(x,t)+h(x,t),$  was studied by weighted homotopy analysis method in \cite{Shidfar}. The time-dependent source function $r(t)$ was determined  in \cite{Hazanee} by utilizing the generalized Fourier approach and the integral measurement. 

Most of the research on parabolic inverse problems has focused on the situation where the observation set $\Gamma$ equals the entire boundary $\partial\Omega$. In \cite{Isakov1993}, Isakov presents the first result in this direction, assuming that $\Gamma = \partial\Omega$ and that one is permitted to measure data at the final time $t = T$ on the complete domain $\partial\Omega$. It is common in physical studies to only have access to extremely small portions of a medium's boundary; therefore, it is fundamentally important to address problems with possibly very small subsets of the full boundary.
 As per our knowledge, there are very few articles related to the different parameter identifications for the parabolic problems
 with the partial boundary measurements.
In \cite{Fei2024}, the author discussed uniqueness results for inverse problems of semilinear reaction-diffusion equations using spherical quasimodes analysis and Runge approximation, with measurements given by the Dirichlet-to-Neumann map on an open subset of $\partial \Omega$. In \cite{Canuto}, the author studied uniqueness by examining the injectivity of the input-output operator for a class of heat equations through boundary measurements, where the measurements are given on the two relatively open pieces  $\Gamma_{in}$ and $\Gamma_{out}$  of the boundary $\partial \Omega$ such that $\Gamma_{in}$ and $\Gamma_{out}$ has a nonempty interior. \par
In the current work, the finite difference method is applied for the computation of the solution of the forward and adjoint problems appearing in the optimization problem. The minimization problem of reconstructing the source function $F(x)$ is formulated as the problem of finding a stationary point of a Lagrangian involving a forward equation (the state equation), a backward equation (the adjoint equation), and the equation expressing that the gradient with respect to the source $F(x)$ vanishes. One can see \cite{Larisa2022,BondestamBeilina} for a description of the optimization approach for the numerical solution of coefficient inverse problems using boundary measurements. \par
Our main contributions of this paper are summarized as follows:
\begin{itemize}
\item We prove a stability estimate for the adjoint problem in the case when the boundary condition is given by the measurements on the part of the boundary.
\item We prove the Fr\'echet differentiability
 of the regularized Tikhonov functional and compute the Fr\'echet derivative, which plays a vital role in the gradient-based numerical algorithm for the inverse problem.
\item  We establish the existence and uniqueness of the solution for the inverse problem when the set of admissible data is bounded.
\item Another significant result is a stability estimate for the inverse source problem (1). Since the inverse problem is posed in the context of a minimization problem, using a first-order necessary optimality condition satisfied by an optimal pair $(u(x,t;F^*),F^*),$
 we establish a local stability estimate for the unknown source term $F^* \in L^2(\Omega)$.
 \item A reconstruction algorithm employing the conjugate gradient method has been developed to determine the source function from partially measured boundary data. Since the step size is critical in gradient-based methods, an iterative formula for determining the step size in the conjugate gradient method is also derived.
 \item To validate the numerical results, we present different numerical examples in 2D and 3D, each with varying noise levels.  
\end{itemize}\par
An outline of the work is as follows. In Section \ref{sec:model} we
  derive the existence results for the direct, adjoint, and inverse problems. In section \ref{sec:fdmopt}, we  present numerical schemes used for FDM discretization of the forward and adjoint problems.
Section \ref{sec:algo}  presents the conjugate gradient algorithm to solve the optimization problem.
Section \ref{sec:numex} shows
numerical examples of the reconstruction of the source function in 2D and 3D
using scattered data of the simulated and exact solutions of the forward problem
collected at the part of the boundary of the computational domain.
Finally, Section \ref{sec:concl} makes conclusions and discusses the obtained results.
\section{Mathematical Analysis of the Direct Problem and Inverse Problem}
\label{sec:model}
\subsection{Function Spaces}
The function spaces and notations listed below are utilized in the subsequent sections. We denote the Sobolev spaces of functions by $L^p(\Omega)$ and $W^{m,p}(\Omega)$ for any $p \in [1, \infty]$ and $m$ is a positive integer, where the corresponding norms are $\|\cdot\|_{L^p(\Omega)}$ and $\|\cdot\|_{W^{m,p}(\Omega)}$. In the case when $p=2$, we use $H^m(\Omega):= W^{m,2}(\Omega)$ and the norm  $\|u\|_{H^m(\Omega)} = {\left(\sum_{|k|\leq m}\|D^ku\|^2_{L^2(\Omega)}\right)}^{\frac{1}{2}}.$ Further, let us denote the standard inner product in $L^2(\Omega)$ by
$(\cdot,\cdot),$ and the inner  product in space and time is denoted by
$((\cdot,\cdot))_{\Omega_T}$.
 
Consider the function spaces $L^2(0,T; H^m(\Omega))$ and $C([0,T];H^m(\Omega)),$  which consist of the functions $u:[0,T]\rightarrow H^m(\Omega)$ such that 
\begin{align*}
\|u\|_{L^2(0,T; H^m(\Omega))} = \left( \int_0^T \|u\|^2_{H^m(\Omega)} dt \right)^{\frac{1}{2}}< \infty,   \\
\|u\|_{C([0,T]; H^m(\Omega))} = \max_{t\in [0,T]}\|u(t)\|_{H^m(\Omega)}< \infty,
\end{align*}
\noindent where $m$ is a non-negative integer. 
We also introduce the following spaces for the Lagrangian formulation of the inverse problem and numerical analysis:
\begin{equation*}\label{spacesfdm}
\begin{split}
H_u^2(\Omega_T) &:= \{ w \in H^2(\Omega_T): w( \cdot , 0) = 0 \}, \\
H_\lambda^2(\Omega_T) &:= \{ w \in  H^2(\Omega_T): w( \cdot , T) = 0\},\\
U^2 &:=H_u^2(\Omega_T)\times H_\lambda ^2(\Omega_T)\times  L^2\left( \Omega \right).
\end{split}
\end{equation*}
We use the below-listed assumptions on the coefficients and source functions:
\begin{assumption}\label{assm}
\begin{equation}
\begin{split}
\begin{cases}
\mathcal F := \{F \in L^2(\Omega): 0<d_1\leq F(x) \leq d_2, \forall x\in \Omega, \,\,\mbox{and} \ \parallel F \parallel _{L^2(\Omega)} \leq d_3 \},\\
a \in C^2(\Omega)~~ \mbox{such that}\,\, 0<a_\min<a(x)<a_\max, \, \forall x\in\Omega \,\,\,\mbox{and}\,\, a_\min,~a_\max\in  \mathbb R^+,  \nonumber \\
G \in L^2(0,T; L^\infty(\Omega)),  \\
z_\delta(x) \in C^\infty(\Omega). 
\end{cases}
\end{split}
\end{equation}
\end{assumption}
\subsection{Stability Analysis for the Direct Problem}
\begin{Theorem}\label{th1}
Assume that the conditions of Assumption \ref{assm} on the functions $a(x)$, $F(x)$ and $G(x,t)$ hold. Then there exists a unique weak solution $u \in L^2(0,T; H^1(\Omega))$, $ \frac{\partial u}{\partial t}\in L^2(0,T; H^1(\Omega)^*)$ of the problem
 (1), which satisfies the following energy estimates:
 \begin{equation}\label{d1}
 \begin{split}
\int_\Omega|\sqrt a u|^2 ~dx + 2\int_0^T\int_\Omega|\nabla u|^2 ~dx dt \leq C_1(a,T)\parallel F \parallel^2_{L^2(\Omega)} \parallel G \parallel^2_{L^2(0,T; L^\infty(\Omega)},
\end{split}
 \end{equation}
where $C_1(a,T) = \left(1+\frac{T}{a_\min}e^{\frac{T}{a_\min}}\right).$\\
Moreover, the weak solution $u$ satisfying \eqref{d1} also has the regularity that 
$u\in L^2(0,T; H^2(\Omega))\cap L^\infty(0,T; H^1(\Omega))$, $\frac{\partial u}{\partial t}\in  L^2(0,T; L^2(\Omega))$ such that the following estimate holds: 
 \begin{eqnarray}\label{d2}
 \lefteqn{\operatorname*{ess\,sup}_{0\leq t\leq T}\| u(t)\|^2_{H^1(\Omega)}          +  \|u\|^2_{L^2(0,T; H^2(\Omega))} + \left\|a\frac{\partial u}{\partial t }\right\|^2_{L^2(0,T; L^2(\Omega)} }\\
 && \leq C_2(a,T) \parallel F \parallel^2_{L^2(\Omega)}\parallel G \parallel^2_{L^2(0,T; L^\infty(\Omega))}, \nonumber
 \end{eqnarray}
where $C_2(a,T)>0$ depends on $a$ and $T$.
\end{Theorem}
\begin{proof}
The existence and uniqueness can be proven by the standard Galerkin approximation (see, \cite{evans2022partial, Salsa}). We only prove the required estimates.
We  multiply equation  \eqref{prb1} by $u$ and integrate over $\Omega$, apply integration by parts and boundary condition $\partial_\nu u = 0$ to get
\begin{equation*}\label{drp3}
\frac{1}{2} \frac{d}{dt}
\int_\Omega  |\sqrt a u|^2~dx + \int_\Omega |\nabla u|^2 ~dx 
 = \int_\Omega  F G u ~dx.  
\end{equation*}
Integrating the obtained equation in time $[0,t]$ and using homogeneous initial conditions, we obtain
\begin{eqnarray} \label{ee}
\frac{1}{2} 
\int_\Omega  |\sqrt a u|^2 ~dx + \int_0^t  \int_\Omega |\nabla u|^2 ~dx dt 
 &=& \int_0^t \int_\Omega  F G u~dx dt \nonumber\\
 &\leq& \|FG\|_{L^2(0,t;L^2(\Omega))} \|u\|_{L^2(0,t;L^2(\Omega))}. 
\end{eqnarray}
From Assumption \ref{assm}, $F\in L^2(\Omega)$ and $G\in L^2(0,T;L^\infty(\Omega))$, we have
\begin{equation}\label{de1}
\|FG\|_{L^2(0,T;L^2(\Omega))}  \leq \|F\|_{L^2(\Omega)}{\left(\int_0^T \|G(t)\|^2_{L^\infty(\Omega)}dt\right)}^{\frac{1}{2}} = \|F\|_{L^2(\Omega)}\|G\|_{L^2(0,T;L^\infty(\Omega))},
\end{equation}
which implies $FG\in L^2(0,T;L^2(\Omega)).$ Then, taking supremum over $t \in [0,T]$  in \eqref{ee},  and using Gr\"onwall's inequality,  we get the following energy estimate 
\begin{equation*}\label{drp3}
\begin{split}
\int_\Omega  |\sqrt a u|^2 ~dx + 2\int_0^T\int_\Omega |\nabla u|^2 ~dx dt &\leq \left(1+ \frac{T}{a_\min}e^{\frac{T}{a_\min}}\right)\parallel F \parallel^2_{L^2(\Omega)}\parallel G \parallel^2_{L^2(0,T; L^\infty(\Omega))}.
\end{split}
\end{equation*}
This completes the proof of \eqref{d1}. For the regularity estimate, multiply \eqref{prb1} by $\frac{\partial u}{\partial t}$ and using the inequality $2ab \leq \frac{1}{\epsilon}a^2 + \epsilon b^2$, one can get
\begin{equation*}
\left(1- \frac{\epsilon}{2a_\min}\right) \int_\Omega \left|\sqrt a\frac{\partial u}{\partial t} \right|^2~dx + \frac{1}{2}\frac{d}{dt}\int_\Omega|\nabla u|^2~dx \leq \frac{1}{2\epsilon} \parallel F \parallel^2_{L^2(\Omega)}\parallel G \parallel^2_{L^2(0,T; L^\infty(\Omega))}.
\end{equation*}
By taking $\epsilon = a_\min,$ rearranging the first term and integrating over $(0,T)$, one can get
\begin{equation}\label{ddd1}
 \int_0^T \int_\Omega\left|a\frac{\partial u}{\partial t} \right|^2~dxdt + a_\max \int_\Omega |\nabla u|^2~dx \leq \frac{a_\max}{a_\min}\parallel F \parallel^2_{L^2(\Omega)}\parallel G \parallel^2_{L^2(0,T; L^\infty(\Omega))}.  
\end{equation}
In particular, from \eqref{de1} and \eqref{ddd1}, we have $FG - a\frac{\partial u}{\partial t} \in L^2(0,T;L^2(\Omega)) $. Now applying elliptic regularity result from \cite{Salsa}, Theorem 8.13, we have
\begin{equation*}\label{ddd2}
\|u(t)\|^2_{H^2(\Omega)} \leq C \left( \parallel F \parallel^2_{L^2(\Omega)}\parallel G \parallel^2_{L^2(0,T; L^\infty(\Omega))} + \int_\Omega\left|a\frac{\partial u}{\partial t} \right|^2~dx  \right).    
\end{equation*}
By integrating over $(0,T)$ and using \eqref{ddd1}, one obtains that
\begin{equation} \label{re}
 \int_0^T \|u(t)\|^2_{H^2(\Omega)}~dt \leq C_2(a,T) \parallel F \parallel^2_{L^2(\Omega)}\parallel G \parallel^2_{L^2(0,T; L^\infty(\Omega))}.  
\end{equation} 
Thus, combining \eqref{ddd1} and \eqref{re}, we complete the estimate \eqref{d2}.  \qed
\end{proof}
\subsection{Lagrangian Approach for the Solution of ISP }
\label{sec:opt1}
This subsection briefly describes
 the reconstruction method  for the inverse source
 problem using the standard optimization approach, which supports, in particular, FDM discretization of the forward and adjoint problems. For details of the Lagrangian approach to the solution of inverse problems, we refer to  \cite{BK,BondestamBeilina}.\par

\noindent\textbf{Inverse Source Problem (ISP).}
Let the space-dependent source function $F \left( x\right)$ in problem (\ref{prb1})-(\ref{prb2}) satisfy Assumption \ref{assm} and 
 is unknown in the domain $\Omega$. We reconstruct $ F \left( x\right) $ in (\ref{prb1})-(\ref{prb2}) for $x\in \Omega$
under the condition that the following  function $ \widetilde u(x,t) $
is known at the boundary $ \partial_1 \Omega$, i.e. at the partial boundary: 
\begin{equation}\label{mes1}
\begin{split}
  u(x,t) = \widetilde u(x,t), ~~\forall  (x,t) \in \partial_1 \Omega \times (0,T).
  \end{split}
\end{equation}
However, due to measurement errors, the output $\widetilde u(x,t)$ always contains a random noise and as a result, an exact equality 
in equation \eqref{mes1} is not possible in practice.\par
To solve  ISP,   we want to find the
 stationary point of the following regularized
 Tikhonov functional
\begin{equation}\label{L1}
J_\gamma(u, F) = \frac{1}{2} \int_\Omega\int_0^T( u(x,t;F) -  \widetilde u(x,t))^2  \delta_{\rm obs} z_\delta (x) ~dxdt   +
\frac{1}{2} \gamma \int_\Omega(F -  F_0)^2(x)~~ dx,
\end{equation}
where $\tilde u(x,t)$ is the observed data at a finite set of observation points $x_{\rm obs} \in \Omega$, the delta-function
$\delta_{\rm obs}$  corresponds to 
the observation points, and $z_\delta $ is a smoothing function (see \cite{BondestamBeilina} for details of choosing this function).\par
To find a minimum of (\ref{L1}), we apply the Lagrangian approach  (see details in
\cite{BK,BondestamBeilina})  and define  the following Lagrangian using the definition of the   forward  model problem \eqref{prb1}:
\begin{equation}\label{lagrangian1fdm}
\begin{split}
  L(v) &= J_\gamma(u, F) +
  ((\lambda, a \frac{\partial u}{ \partial t } - \triangle u - F(x)G(x,t)))_{\Omega_T},
\end{split}
\end{equation}
where
$v=(u,\lambda, F) \in U^2$. We now search for a stationary point of the Lagrangian
with respect to $v$ satisfying for all
$\bar v= (\bar u, \bar\lambda, \bar F) \in U^2$
 to the following optimality condition
\begin{equation}
 L^\prime(v; \bar v) = 0 ,  \label{scalar_lagr1fdm}
\end{equation}
where $ L^\prime (v; \bar v )$
is the Fr\'{e}chet derivative  of
the Lagrangian (\ref{lagrangian1fdm}) or
Jacobian of $L$ at $v$.
The optimality condition \eqref{scalar_lagr1fdm} for the Lagrangian (\ref{lagrangian1fdm})
 can also be written  for all
$\bar v \in U^2$ as
\begin{equation}
  L^\prime(v; \bar{v}) =
  L^\prime_\lambda(v; \bar\lambda)  + L^\prime_u(v; \bar u) + L^\prime_F(v; \bar F): =
  \frac{\partial L}{\partial \lambda}(v)(\bar\lambda) +  \frac{\partial L}{\partial u}(v)(\bar u)
  + \frac{\partial L}{\partial F}(v)(\bar F) = 0.  \label{scalar_lagrang1fdm}
\end{equation}
We observe from   \eqref{scalar_lagrang1fdm} 
 that
to   satisfy optimality condition \eqref{scalar_lagr1fdm},
every component of   \eqref{scalar_lagrang1fdm} should be zero.
To find the Fr\'{e}chet derivative (\ref{scalar_lagr1fdm}) of
the Lagrangian (\ref{lagrangian1fdm}), we consider $L(v + \bar v) - L(v),~
\forall \bar v \in U^2$ and single out the linear part of the
obtained expression with respect to $\bar v$, and ignoring  all nonlinear terms.  In the derivation of the
Fr\'{e}chet derivative we assume that in the Lagrangian
(\ref{lagrangian1fdm}) variables in $v=(u,\lambda,  F) \in U^2$ can be
varied independent of each other and thus 
the Fr\'{e}chet derivative of the Lagrangian (\ref{lagrangian1fdm})  will be the same
as by assuming that functions $u$ and $\lambda$ are dependent on the
 source function $F$, see discussion in Chapter 4 of \cite{BK}.
 The optimality conditions \eqref{scalar_lagr1fdm} for the Lagrangian (\ref{lagrangian1fdm})
 for all
 $\bar v \in U^2$ are derived in  several  works, see 
 \cite{BK,BondestamBeilina}
 for more details of the derivation.
 More precisely, the optimality conditions \eqref{scalar_lagr1fdm} for the Lagrangian (\ref{lagrangian1fdm})
 for all
 $\bar v \in U^2$   are: 
\begin{equation*}\label{forward1fdm}
\begin{split}
0 &= L^\prime_\lambda(v; \bar\lambda) =  \frac{\partial L}{\partial \lambda}(v)(\bar{\lambda}) =
((a \frac{\partial u}{ \partial t },   \bar\lambda))_{\Omega_T}\\
&- \left(\left( \triangle u, \bar\lambda\right)\right)_{\Omega_T} 
-  \left(\left(  F(x)G(x,t), \bar{\lambda}\right)\right)_{\Omega_T}  
,~~\forall \bar\lambda \in H_\lambda^2(\Omega_T),
\end{split}
\end{equation*}
\begin{equation*} \label{control1fdm}
\begin{split}
0 &=   L^\prime_u(v; \bar u) =  \frac{\partial L}{\partial u}(v)(\bar u) = 
 \int_\Omega\int_0^T( u - \widetilde u)~ \bar u  ~\delta_{\rm obs} ~ z_\delta(x)~ dx dt \\
&- ((a \frac{\partial \lambda}{ \partial t },   \bar u ))_{\Omega_T}
 - \left(\left(  \Delta \lambda,   \bar u \right)\right)_{\Omega_T} + \int_{\partial \Omega}\int_0^T \partial_\nu\lambda~\bar u~ dx dt, \quad
~\forall \bar u \in H_u^2(\Omega_T),
\end{split}
\end{equation*} 
\begin{equation*} \label{grad1new}
\begin{split}
0 &=  L^\prime_F(v; \bar F) =  \frac{\partial L}{\partial F}(v)(\bar F) 
=  - \left(G(x,t) \lambda,  \bar F \right)  + \gamma \left(( F-  F_0), \bar F \right),~\forall \bar F \in L^2(\Omega).
\end{split}
\end{equation*}
\noindent We observe  that  the condition $L^\prime_\lambda(v; \bar\lambda) = 0 $
 corresponds to the state equation
(\ref{prb1}), and  the condition  $L^\prime_u(v; \bar u) = 0 $ will result in the
 following adjoint problem
\begin{equation}\label{adj11}
\begin{cases}
-a(x) \frac{\partial \lambda}{\partial t} - \Delta \lambda =  -(u(x,t) - \tilde u(x,t)) \delta_{\rm obs} z_\delta(x), \quad (x,t)\in \Omega\times(0,T),\\
\lambda(x,T) = 0, \quad x\in \Omega, \\
\partial_\nu\lambda(x,t) = 0, \quad (x,t)\in \partial\Omega\times (0,T).  
\end{cases}
\end{equation}
\subsection{Stability Analysis for the Adjoint Problem}
In this subsection, we discuss the stability estimate for the adjoint system \eqref{adj11} of the direct problem \eqref{prb1}.
\begin{Theorem}\label{th2}
Assume that conditions from Assumption \ref{assm} on the functions $a(x)$, $z_\delta(x)$ hold and the solution of direct problem $u\in L^2(0,T; L^2(\Omega))$. Then there exists a unique weak solution $\lambda \in L^2(0,T; H^1(\Omega))$ such that $\frac{\partial \lambda}{\partial t} \in L^2(0,T; (H^1(\Omega))^*)$ of the problem \eqref{adj11}, which satisfies the following energy estimate:
\begin{equation}\label{ad1}
\begin{split}
 \int_\Omega|\sqrt a\lambda|^2 ~dx + 2\int_0^T\int_\Omega|\nabla \lambda|^2~dx dt \leq C_1(a,T)\parallel (u-\tilde u) \delta_{\rm obs} z_\delta  \parallel^2_{L^2(0,T; L^2(\Omega))},
 \end{split}
\end{equation}
 where $C_1(a,T)$ is given in Theorem \ref{th1}.
Moreover, the weak solution $\lambda$ of \eqref{adj11} also satisfies the following regularity estimate: 
 \begin{eqnarray}\label{d11}
 \lefteqn{\operatorname*{ess\,sup}_{0\leq t\leq T}{\| \lambda(t)\|}^2_{H^1(\Omega)}          +  \|\lambda\|^2_{L^2(0,T;H^2(\Omega))} + \left\|a\frac{\partial \lambda}{\partial t}\right\|^2_{L^2(0,T;L^2(\Omega)}}  \nonumber \\    
&&\leq C_3(a,T)\parallel(u-\tilde u) \delta_{\rm obs} z_\delta \parallel^2_{L^2(0,T; L^2(\Omega))},
 \end{eqnarray}
where $C_3(a,T)>0$ depends on $a$ and $T$.
\end{Theorem}
\begin{proof}
We only sketch the proof of the estimate \eqref{ad1}.
Multiply the equation  \eqref{adj11} by $\lambda$ and integrate over $\Omega$, apply integration by parts and the boundary condition $\partial_\nu \lambda = 0$ to get
\begin{eqnarray}\label{drp34}
\lefteqn{- \frac{d}{dt}
 \int_\Omega  |\sqrt a \lambda|^2 ~dx + 2\int_\Omega |\nabla \lambda|^2 ~dx }\\
 &&\leq \int_\Omega |(u(x,t) - \tilde u(x,t)) \delta_{\rm obs} z_\delta(x)|^{2}~dx + \frac{1}{a_{\min}}\int_{\Omega} |\sqrt a \lambda|^{2}~dx.  \nonumber
\end{eqnarray}
Now using Gr\"onwall's inequality, we have 
\begin{equation}\label{drp35}
\int_\Omega  |\sqrt a \lambda|^2 ~dx  
 \leq e^{\frac{(T-t)}{a_{\min}}} \parallel(u-\tilde u)  \delta_{\rm obs} z_\delta  \parallel^2_{L^2(t,T; L^2(\Omega))} .  
\end{equation}
By integrating  \eqref{drp34} over $[t,T]$, using the terminal condition  $\lambda(x,T)=0$ and  \eqref{drp35} on the right-hand side of  \eqref{drp34}, we obtain
\begin{equation*}\label{drp3}
\begin{split}
\int_\Omega  |\sqrt a \lambda|^2 ~dx +
2\int_t^T\int_\Omega |\nabla \lambda |^2 ~dx dt \leq 
 & \left(1+ \frac{(T-t)}{a_\min}e^{\frac{(T-t)}{a_\min}}\right) \parallel(u-\tilde u )  \delta_{\rm obs} z_\delta  \parallel^2_{L^2(t,T; L^2(\Omega))}.  
\end{split}
\end{equation*}
\noindent By taking  $t\rightarrow 0$ in the above inequality, we get the estimate
\eqref{ad1}.
The improved regularity result directly follows from Theorem \ref{th1} with the only difference being the change
 on the right-hand side  of the estimate \eqref{d11}.\qed
\end{proof}
\subsection{Gradient of the Functional}
This section deals with the derivation of the Fr\'{e}chet derivative of the objective functional $J_\gamma(u,F)$.
 Assume that the coefficients $F, F+\delta F \in \mathcal{F}$.
Let $\delta u = u(x,t,F+\delta F)-u(x,t,F)$.
Consider the increment in the functional \eqref{L1} as follows:
\begin{equation*} \label{gr1}
\begin{split}
\delta J_\gamma(u,F) &= J_\gamma(u,F+\delta F) - J_\gamma(u,F)\\
& = \frac{1}{2}\int_0^T\int_\Omega(\delta u)^2
\delta_{\rm obs} z_\delta(x)~dx dt +
\int_0^T\int_\Omega\delta u(u-\tilde u)\delta_{\rm obs}z_\delta(x) ~dx dt\\
& \ \ \ + \frac{\gamma}{2}\int_\Omega(\delta F(x))^2 ~dx + \gamma \int_\Omega\delta F(x)(F-F_0)(x)~dx. 
\end{split}
\end{equation*}
\begin{Theorem}
Suppose  Assumption \ref{assm} holds and let $\lambda$ be the weak solution to the adjoint problem \eqref{adj11}.  Then the functional $J_\gamma(u,F)$ is Fr\'echet differentiable with the derivative 
\begin{equation} 
  J^\prime_\gamma(u,F)(x) = - \int_0^TG(x,\tau)\lambda(x,\tau) d\tau + \gamma(F-F_0)(x). \label{gra}
\end{equation}
\end{Theorem}
\begin{proof}
From the direct problem \eqref{prb1}, the function $\delta u$ satisfies the following equation
\begin{equation}\label{grad1}
\begin{cases}
a \frac{\partial(\delta u)}{\partial t} - \Delta{\delta u} = \delta F(x)G(x,t) , \quad (x,t)\in \Omega\times(0,T) \\
\delta u(x,0) = 0, \quad x\in \Omega,\\
\partial_\nu\delta u(x,t) = 0, \quad (x,t)\in \partial\Omega\times(0,T).    
\end{cases}
\end{equation}
We multiply the equation  \eqref{grad1} by $\lambda$ and the  equation
 \eqref{adj11} by $\delta u$, then integrate  over $\Omega_{T}$
 to get:
\begin{equation}\label{fr2}
\int_0^T\int_\Omega a \lambda \frac{\partial(\delta u)}{\partial t}~dxdt
- \int_0^T\int_\Omega \Delta\delta u\lambda~dxdt = \int_0^T\int_\Omega \delta F(x)G(x,t) \lambda~dxdt
\end{equation}
and \begin{eqnarray}\label{fr3}
- \int_0^T\int_\Omega  a \frac{\partial\lambda}{\partial t}\delta u ~dxdt
 -\int_0^T\int_\Omega \Delta\lambda\delta u ~dxdt= - \int_0^T\int_\Omega\delta u(u(x,t)- \tilde u(x,t))  \delta_{\rm obs} z_\delta(x)~dxdt.
\end{eqnarray}
Subtracting \eqref{fr2} from \eqref{fr3}, after doing integration by parts on the left-hand side and using initial and final conditions, we get
\begin{equation}\label{fr4} 
-\int_0^T\int_\Omega(u(x,t)-\tilde u(x,t)) \delta_{\rm obs}z_\delta(x)~  \delta u~dx dt = \int_0^T\int_\Omega\delta F(x)G(x,t)\lambda(x,t)~dx dt.
\end{equation}
Now, using the equation \eqref{fr4} in  $\delta{J_\gamma}(u,F),$ we have
\begin{align*}
\delta J_\gamma(u,F) &=  \frac{1}{2}\int_0^T\int_\Omega(\delta u)^2 \delta_{\rm obs}z_\delta(x) ~dx dt + \int_0^T\int_\Omega - \delta F(x)G(x,t)\lambda(x,t)~dx dt\\
& + \frac{\gamma}{2}\int_\Omega(\delta F(x))^2 ~dx + \gamma \int_\Omega\delta F(x)(F-F_0)(x)~dx
\end{align*}
and
\begin{align*}
& \left| \delta J_\gamma(u,F) - \int_\Omega\left\{ \left( - \int_0^TG(x,t)\lambda(x,t) ~dt\right) + \gamma(F-F_0)(x)\right\}\delta F(x)~dx\right| \\ 
& \leq \frac{1}{2}\int_0^T\int_\Omega (\delta u)^2|\delta_{\rm obs}z_\delta(x)| ~dx dt + \frac{\gamma}{2}\int_\Omega(\delta F(x))^2 ~dx.
\end{align*}
Now, using the stability estimate  \eqref{d1}, which holds for $\delta u,$   we get
\begin{equation*}\label{fr11}
\begin{split}
& \left| \delta J_\gamma(u,F) - \int_\Omega\left\{ \left(- \int_0^T G(x,t)\lambda(x,t) dt\right) + \gamma(F-F_0)(x)\right\}\delta F(x)dx\right|\\
& \qquad \leq  O(\parallel \delta F \parallel^2_{L^2(\Omega)}).
\end{split}
\end{equation*}
Hence, the functional $J_\gamma(u,F)$ is Fr\'echet differentiable and the Fr\'echet derivative $J^\prime_\gamma(u,F)(x)$ is given by \eqref{gra}. Hence the proof.  \qed
\end{proof}
\subsection{Existence of Minimizer for the Functional}
\begin{Theorem}
Suppose Assumption \ref{assm} holds true. Then, for $\gamma>0$, there exists a unique minimizer of the functional $J_\gamma(u,F)$.
\end{Theorem}
\begin{proof}
Since the functional $J_\gamma(u,F)$ is bounded below, one can argue that there exists a minimizing sequence $\{F_n\}\in \mathcal F$ converges weakly to an admissible source $F\in \mathcal F$.  Note that
 the following identity holds for the non regularized Tikhonov functional $J(u,F)$ which can be obtained by replacing $\gamma=0$ in \eqref{L1}:
 \begin{equation}\label{in1}
 \begin{split}
& J(u,F) = J(u,F_n) - \frac{1}{2}\int_0^T\int_\Omega[u(x,t;F_n) - u(x,t;F)]^2\delta_{\rm obs} z_\delta(x) ~dxdt \\
 & - \int_0^T\int_\Omega [u(x,t;F) - \tilde u(x,t)]\delta_{\rm obs}z_\delta(x)\delta u ~dxdt,
 \end{split}
 \end{equation}
where $\delta u = u(x,t;F_n) - u(x,t;F)$. Replacing the last integral of \eqref{in1} by the identity in \eqref{fr4}, we have
 \begin{equation*}\label{in2}
 \begin{split}
 J(u,F) \leq J(u,F_n) + \int_0^T\int_\Omega\delta F(x)G(x,t)\lambda(x,t) ~dxdt,
 \end{split}
 \end{equation*}
 where $\delta F = F_n-F$.
 By H\"older's inequality, we have $\int_0^T\lambda(x,\cdot)G(x,\cdot)dx \in L^{2}(\Omega)$.
Since $F_n\ \rightharpoonup F$ in $L^2(\Omega)$, we have
\begin{equation*}
\begin{split}
\int_\Omega\left(\int_0^T\lambda(x,t)G(x,t)dt\right)F_n(x) ~dx \rightarrow \int_\Omega\left(\int_0^T\lambda(x,t)G(x,t)dt\right)F(x)~dx, \quad \mbox{as} \quad n\rightarrow \infty.
\end{split}
\end{equation*}
This implies that $J(u,F) \leq \lim_{n\rightarrow \infty}J(u,F_n).$ Now, as $F_n\ \rightharpoonup F$ in $L^2(\Omega)$,  it is lower semi-continuous, that is, $\parallel F \parallel_{L^2(\Omega)}\leq \lim \inf_{n\rightarrow \infty}\parallel F_n \parallel_{L^2(\Omega)}$, so it will implies that $J_\gamma(u,F) \leq  \lim \inf_{n\rightarrow \infty} J_\gamma(u,F_n)$  in $\mathcal F$, hence $J_\gamma(u,F)$ is lower semi-continuous. Moreover, since the inverse problem is linear, the functional $J_\gamma(u,F)$ is strictly convex.
So, using the generalized Weierstrass theorem (see \cite{zeidler2012applied}, Theorem 2D), we conclude that the regularized functional $J_\gamma(u,F)$ has a unique minimizer. Hence, the proof. \qed
 \end{proof}
%\begin{Remark} By careful inspection, we note that if the set of admissible class of source functions is unbounded then, with the help of H\"older's inequality, we can prove that functional $J_\gamma(u,F)$ is weakly coercive. That is, as $\parallel F \parallel_{L^2(\Omega)}\rightarrow \infty$ $\implies$ $\parallel F-F_0 \parallel^2_{L^2(\Omega)} \rightarrow \infty$. Since $J_\gamma(u,F) \geq \frac{\gamma}{2}\parallel F-F_0 \parallel^2_{L^2(\Omega)}$, as $\parallel F \parallel_{L^2(\Omega)}\rightarrow \infty$, functional $J_\gamma(u,F)$ is weakly coercive. Hence, the functional $J_\gamma(u,F)$ has a minimizer.  
%\end{Remark}
\subsection{Stability Analysis for the Inverse Problem}
\begin{Theorem}(Variational inequality)
Suppose $u$ and $\lambda$ are the solutions of \eqref{prb1} and \eqref{ad1} respectively, and $F^*$ be the solution to the minimum problem of \eqref{L1}. Then the following variational inequality holds:
\begin{eqnarray}\label{stb1}
\lefteqn{ -\int_0^T\int_\Omega(k(x) - F^*(x))G(x,t)\lambda(x,t) ~dx dt }\\
 &&+ \gamma \int_\Omega(F^*(x) - F_0(x))(k(x)-F^*(x))~dx \geq 0, \ \  \forall k\in \mathcal F.\nonumber
\end{eqnarray}
\end{Theorem}
\begin{proof}
 Let $k\in \mathcal F,$  $0\leq \beta \leq 1$ and $F_\beta(x) = F^*(x) + \beta(k(x)-F^*(x))\in \mathcal F.$  Let $(u_\beta,F_\beta)$ be the solution of the direct problem.  Then the corresponding objective functional is 
 \begin{equation*}\label{st2}
 \begin{split}
 J_\gamma(u_\beta,F_\beta) = \frac{1}{2}\int_\Omega\int_0^T(u_\beta(x,t) - \tilde u(x,t))^2  \delta_{\rm obs} z_\delta(x)  ~dx dt + \frac{\gamma}{2}\int_\Omega(F_\beta(x)-F_0(x))^2~dx.  
 \end{split}
 \end{equation*} 
Since $J_\gamma(u_\beta,F_\beta)$ is Fr\'echet differentiable, we have
\begin{equation}\label{st5}
\begin{split}
&\frac{d}{d\beta}J_\gamma(u_\beta,F_\beta)\Big |_{\beta =0} =   \int_\Omega\int_0^T \left[(u_\beta(x,t) - \tilde u(x,t))  \delta_{\rm obs} z_\delta(x) \frac{\partial u_\beta}{\partial \beta}\right]_{\beta=0}~dxdt\\
& \qquad  + \gamma \int_\Omega(F_\beta(x) - F_0(x))\Big |_{\beta=0}(k(x)-F^*(x))~dx.
\end{split}
\end{equation}
By taking $\eta =\frac{\partial u_\beta}{\partial \beta}\Big |_{\beta=0}$ and $u = u_\beta\Big |_{\beta=0}$, we obtain the following system satisfied by $\eta$:
\begin{equation}\label{sens1}
\begin{split}
\begin{cases}
a(x)\eta_t - \Delta\eta = (k(x) - F^*(x))G(x,t), \quad (x,t)\in \Omega\times(0,T),\\
\eta(x,0) = 0, \quad  x\in\Omega, \\
\partial_\nu\eta(x,t) = 0, \quad (x,t)\in \partial\Omega\times(0,T).
\end{cases}
\end{split}
\end{equation}   
\noindent Since $F^*$ is an optimal solution, we have
\begin{align*}
\begin{split}
\frac{d}{d\beta}J_\gamma\left(u_\beta,F^*(x)+\beta(K - F^*)(x)\right)\Big |_{\beta=0}\geq 0.
\end{split}
\end{align*}
From \eqref{st5}, we get the inequality 
\begin{eqnarray}\label{st6}
\int_\Omega\int_0^T(u(x,t) - \tilde u(x,t))  \delta_{\rm obs} z_\delta(x)\eta(x,t)~dx dt + \gamma \int_\Omega(F^*(x) - F_0(x))(k(x)-F^*(x))~dx\geq 0.  
\end{eqnarray}
Multiply system \eqref{sens1} by $\lambda,$ the solution of adjoint system \ref{adj11} and integrating by parts, we get
\begin{eqnarray}\label{st7}
-\int_\Omega\int_0^T(u(x,t) - \tilde u(x,t))  \delta_{\rm obs} z_\delta(x) \eta(x,t)~dx dt = \int_0^T\int_\Omega(k(x)-F^*(x))G(x,t)\lambda(x,t)~dxdt. 
\end{eqnarray}
Substituting \eqref{st7} in \eqref{st6}, we get the variational inequality \eqref{stb1}. \qed \\
\end{proof}

Let $(u,F)$ and $(\bar u,\bar F)$ be two solutions of the direct problem \eqref{prb1}. Then $U = u - \bar u$ and $\hat F = F - \bar F$ satisfy the following system
\begin{equation}\label{sens2}
\begin{cases}
a(x)\frac{\partial U}{\partial t} - \Delta U = \hat F(x)G(x,t), \quad (x,t)\in\Omega \times (0,T),\\
U(x,0) =0, \quad x\in \Omega, \\
\partial_\nu U(x,t) = 0, \quad \partial\Omega\times(0,T).   
\end{cases}
\end{equation}
Similarly, let $\lambda$ and $\bar\lambda$ be the two adjoint
 solutions corresponding to $(u,F)$ and $(\bar u,\bar F)$ respectively, then $\Lambda = \lambda - \bar\lambda$ will satisfies the following system 
 \begin{equation}\label{sens3}
\begin{cases}
 -a(x)\frac{\partial \Lambda}{\partial t} - \Delta\Lambda = -\left((u-\bar u)-(\tilde u- \tilde{\bar u})\right) \delta_{\rm obs} z_\delta(x), \quad (x,t)\in \Omega \times (0,T), \\
 \Lambda(x,T) = 0, \quad x\in \Omega,  \\
 \partial_\nu\Lambda(x,t) = 0, \quad (x,t)\in \partial\Omega\times(0,T),\\
\end{cases}
 \end{equation}
where $\tilde{u}$ and $\tilde{\bar u}$ are the given partial boundary measurements.
Next, we derive  \emph{a priori} estimates for the equations \eqref{sens2} and \eqref{sens3} which are
 essential to prove the stability result.
\begin{Lemma}
Assume $F$, $\bar F\in \mathcal F$ and let $U$ be the solution of the system \eqref{sens2}. 
Then, we have  
\begin{equation}\label{stb8}
\int_\Omega|\sqrt a U|^2 dx + 2\int_0^T\int_\Omega|\nabla U|^2 dx dt \leq C_1(a,T)\parallel \hat F \parallel ^2_{L^2(\Omega)} \parallel G \parallel^2_{L^2(0,T; L^\infty(\Omega)}.  
\end{equation}
where $C_1(a,T)$ is given in Theorem \ref{th1}.
\end{Lemma}
\begin{proof}
The proof follows from the same line of arguments of Theorem \ref{th1}. \qed
\end{proof}
\begin{Lemma}\label{stb9}
Let $\Lambda$ be the solution of the system \eqref{sens3}, then the following estimate holds
 \begin{equation}\label{stb10}
 \begin{split}
 &\int_\Omega|\sqrt a\Lambda|^2 ~dx + 2\int_0^T\int_\Omega|\nabla \Lambda|^2~dxdt \leq  C_1(a,T) \\
 & \times 2\left[\parallel (u-\bar u) \delta_{\rm obs}z_\delta \parallel^2_{L^2(0,T; L^2(\Omega))}
  + \parallel (\tilde u-\tilde{\bar u})\delta_{\rm obs} z_\delta \parallel^2_{L^2(0,T; L^2(\Omega))}\right], 
 \end{split}
 \end{equation}
 where $C_1(a,T)$ is given in Theorem \ref{th1}.
\end{Lemma}
\begin{proof}
The proof is again a direct consequence of Theorem\ref{th2}, since the right-hand side of the system \eqref{sens3} can be estimated as follows:
\begin{eqnarray*}
\lefteqn{\parallel ((u-{\bar u})- (\tilde u-\tilde{\bar u}))\delta_{\rm obs} z_\delta(x) \parallel^2_{L^2(0,T; L^2(\Omega))}} \\
\ && \leq  \parallel (u-\bar u) \delta_{\rm obs} z_\delta(x)   \parallel^2_{L^2(0,T; L^2(\Omega))} 
  +  \parallel (\tilde u-\tilde{\bar u}) \delta_{\rm obs} z_\delta(x) \parallel^2_{L^2(0,T; L^2(\Omega))}.  \qed
\end{eqnarray*} 
\end{proof}
\begin{Theorem}\label{stabt}
Let $F$ and $\bar F$ be the unique minimizers of the functional $J_\gamma(u,F)$ corresponding to the measurements $\bar{u}$ and $\tilde{\bar{u}}$ respectively. Then there exists a time $T_0>0$ and a constant $L(T_0)>0$ such that the following stability estimate holds:
 \begin{eqnarray}\label{stb110}
 \parallel F-\bar F\parallel_{L^2(\Omega)} \leq \frac{2(L(T_0))^{\frac{1}{2}}}{\gamma}\parallel G \parallel_{L^2(0,T; L^\infty(\Omega))}   \parallel (\bar u(x,t)-\tilde{\bar u}(x,t)) \delta_{\rm obs}z_\delta(x) \parallel_{L^2(0,T; L^2(\Omega))},
 \end{eqnarray}
where
$L(T) = \frac{T C_1(a,T)}{a_\min}$ and $C_1(a,T)$ is given in Theorem \ref{th1}.
\end{Theorem}
\begin{proof}
Replace $k$ by $\bar F$ and $F^*$ by $F$ in the variational inequality \eqref{stb1}, we get
\begin{eqnarray}\label{stb14}
-\int_0^T\int_\Omega(\bar F(x)-F(x))G(x,t)\lambda(x,t) ~dx dt + \gamma \int_\Omega(F(x)-F_0(x))(\bar F(x) - F(x)) ~dx \geq 0.   
\end{eqnarray}
Now, again changing $k$ by $F$ and $F^*$ by $\bar F$ in \eqref{stb1}, we have 
\begin{eqnarray}\label{stb15}
 -\int_0^T\int_\Omega(F(x)-\bar F(x))G(x,t)\bar{\lambda}(x,t) ~dx dt + \gamma \int_\Omega(\bar F(x)-F_0(x))(F(x) - \bar F(x)) ~dx \geq 0.   
\end{eqnarray}
By adding \eqref{stb14} and \eqref{stb15}, one can get 
\begin{equation*} \label{min}
\gamma \int_\Omega(F(x)-\bar F(x))^2 dx \leq \int_0^T\int_\Omega(F(x)-\bar F(x))G(x,t)(\lambda-\bar\lambda)(x,t) dx dt.    
\end{equation*}
Now, applying H\"older's inequality and using the inequality \eqref{de1}, we have 
\begin{eqnarray}\label{stb77}
 \gamma \int_\Omega(F(x)-\bar F (x))^2~dx &\leq & \|(F-\bar F)G\|_{L^2(0,T;L^2(\Omega))}\|\lambda-\bar\lambda\|_{L^2(0,T;L^2(\Omega))} \nonumber \\
  &\leq&  \parallel F - \bar F\parallel_{L^2(\Omega)} \parallel G \parallel_{L^2(0,T; L^\infty(\Omega)}\parallel \Lambda \parallel_{L^2(0,T;L^2(\Omega))}, 
\end{eqnarray}
where $\Lambda = \lambda - \bar{\lambda}.$
The estimate \eqref{stb10} leads to the following: 
\begin{eqnarray}\label{stb41}
\parallel \Lambda \parallel^2_{L^2(0,T;L^2(\Omega))} &\leq &\frac{2T}{a_{\min}}\left(1+ \frac{T}{a_\min}e^{\frac{T}{a_\min}}\right)\left[\parallel \delta_{\rm obs}z_{\delta}(x)\parallel^2_{L^\infty(\Omega)}\parallel U \parallel^2_{L^2(0,T;L^2(\Omega))} \right.\\
 && \left.+\parallel (\tilde u(x,t)-\tilde{\bar u}(x,t)) \delta_{\rm obs} z_\delta(x)\parallel^2_{L^2(0,T; L^2(\Omega))}\right], \nonumber
\end{eqnarray}
where $U = u- \bar{u}.$ Now from estimate \eqref{stb8}, we have 
\begin{eqnarray}\label{stb51}
\parallel U \parallel^2_{L^2(0,T;L^2(\Omega))} \leq \frac{T C_{1}(a,T)}{a_\min}\parallel \hat{F} \parallel^2_{L^2(\Omega)} \parallel G \parallel^2_{L^2(0,T; L^\infty(\Omega))},  
\end{eqnarray}
where $\hat F = F- \bar{F}.$
Now using  \eqref{stb51} in  \eqref{stb41}, we get 
\begin{eqnarray}\label{stb78}
\parallel \Lambda \parallel^2_{L^2(0,T;L^2(\Omega))} &\leq& \frac{2T^2 (C_1(a,T))^2}{a^2_\min} \parallel \hat{F} \parallel^2_{L^2(\Omega)} \parallel G \parallel^2_{L^2(0,T; L^\infty(\Omega))}\parallel \delta_{\rm obs}z_{\delta}(x)\parallel^2_{L^\infty(\Omega)}  \nonumber \\
&& + \frac{2T C_1(a,T)}{a_\min}\parallel (\tilde u(x,t)-\tilde{\bar u}(x,t)) \delta_{\rm obs} z_\delta(x)\parallel^2_{L^2(0,T; L^2(\Omega))}.      
\end{eqnarray}
Using the estimate \eqref{stb78} in the estimate \eqref{stb77}, one can get
\begin{eqnarray*}
\gamma^2 \parallel F-\bar F \parallel^2_{L^2(\Omega)} &\leq& \frac{2T^2\left(C_1(a,T)\right)^2}{a^2_\min}\parallel F-\bar F \parallel^2_{L^2(\Omega)}  \parallel G \parallel^4_{L^2(0,T;L^\infty(\Omega))}\|\delta_{\rm obs} z_\delta(x) \|^2_{L^\infty(\Omega)} \\
 & +& \frac{2TC_1(a,T)}{a_\min} \parallel (\tilde u(x,t)-\tilde{\bar u}(x,t)) \delta_{\rm obs} z_\delta(x)\parallel^2_{L^2(0,T; L^2(\Omega))}\parallel G \parallel^2_{L^2(0,T;L^\infty(\Omega))}.
\end{eqnarray*}
\noindent Now let, \begin{align*}
L(T) := \frac{T C_1(a,T)}{a_\min}, \,\,K(T) := 2(L(T))^2 \parallel G \parallel^4_{L^2(0,T; L^\infty(\Omega))}\| \delta_{\rm obs} z_\delta(x) \|^2_{L^\infty(\Omega)}. 
 \end{align*}
\noindent Choosing $T_0$ such that $\frac{K(T_0)}{\gamma^2}\leq \frac{1}{2}$, we get
 \begin{equation*}\label{stb111}
 \begin{split}
 &\parallel F-\bar F \parallel_{L^2(\Omega)} \leq \frac{2(L(T_0))^{\frac{1}{2}}}{\gamma}\parallel G \parallel_{L^2(0,T; L^\infty(\Omega))}   \parallel (\tilde u(x,t)-\tilde{\bar u}(x,t)) \delta_{\rm obs} z_\delta(x) \parallel_{L^2(0,T; L^2(\Omega))}. 
  \end{split}
 \end{equation*} 
 This completes the proof. \qed
\end{proof}
\begin{corollary}
Suppose the conditions of Theorem \ref{stabt} hold. If the given measurements are unique, that is, $\bar u(x,t) = \tilde{\bar u}(x,t)$ for a.e. $(x,t)\in \partial_1\Omega\times (0,T)$, then there exists a time instant $T_0$ such that for $T\geq T_0$, we have $F(x) = \bar F(x)$ for a.e. $x\in \Omega.$ 
\end{corollary}
\begin{proof}
The estimate in equation \eqref{stb110} clearly shows that when $\tilde u(x,t) =\tilde{\bar u}(x,t)$ for a.e. $(x,t)\in \partial_1\Omega\times (0,T)$, it follows that $F(x) = \bar F(x)$ for a.e. $x\in \Omega$. \qed
\end{proof}
 \begin{Remark}  From Theorem \ref{stabt} above, we can observe that how the regularization parameter $\gamma$, the minimum value $a_\min$, and the final time $T$ directly affect the stability constant. This suggests that if $a_\min$ or $\gamma$ is sufficiently large, or if the final time $T$ is too small, the stability constant becomes a very small. Further, the smallness condition on $T$ can be modified by choosing $\gamma$ sufficiently large, that is, $\sqrt{2K(T)} \leq \gamma.$
 \end{Remark}
 \section{Numerical Reconstruction of Source Term}
\subsection{Finite Difference Discretization}
\label{sec:fdmopt}
For computations, we discretize $\Omega_T$ in space and time.
Let $J_{\tau}$ denotes discretization of the time interval $[0,T]$ into time
 sub-intervals $J=(t_{k-1},t_k]$ of  the  length
   $\tau=T/N$. Here,  $N$ is the number of time intervals.\par
To discretize the space $\Omega$
 we denote by $K_h = \{ K\}$ a partition of the domain $\Omega$
 into cubes in $ \mathbb R^3 $ or squares in $ \mathbb R^2$.
We choose the mesh sizes $\Delta x=\frac{1}{N_x}, \Delta y = \frac{1}{N_y}, \Delta z = \frac{1}{N_z}$  where $N_x,N_y,N_z$ are the number of mesh points in $x,y,z$ directions, respectively.\par
For the numerical solutions to the forward and adjoint problems, we use
 the Cranck-Nicolson
 method, which is an efficient method for the solution of a parabolic PDE. This scheme for the forward problem
\eqref{prb1} leads to the following discrete form
for 2D case:
\begin{equation}\label{dis1}
\begin{cases}
&a_{i,j}\frac{u^{n+1}_{i,j} - u^n_{i,j}}{\Delta t} = \frac{1}{2}\left(\frac{u^{n+1}_{i-1,j} - 2u^{n+1}_{i,j} + u^{n+1}_{i+1,j}}{\Delta x^2} + \frac{u^{n+1}_{i,j-1} - 2u^{n+1}_{i,j} + u^{n+1}_{i,j+1}}{\Delta y^2}\right) \\
&+ \frac{1}{2}\left(\frac{u^n_{i-1,j} - 2u^n_{i,j} + u^n_{i+1,j}}{\Delta x^2} + \frac{u^n_{i,j-1} - 2u^n_{i,j} + u^n_{i,j+1}}{\Delta y^2} \right) + \frac{1}{2} F_{i,j} G^{n+1}_{i,j} + \frac{1}{2} F_{i,j} G^n_{i,j}.
\end{cases}
\end{equation}
%Further, simplified form of \eqref{dis1} for the 2D case is as follows:
Further,  rearranged form of \eqref{dis1} for the 2D case is as follows:
\begin{equation*}\label{dis2}
\begin{cases}
&a_{i,j}u^{n+1}_{i,j} - \frac{1}{2}\left[ \frac{\Delta t}{\Delta x^2}(u^{n+1}_{i+1,j} - 2u^{n+1}_{i,j} + u^{n+1}_{i-1,j}) + \frac{\Delta t}{\Delta y^2}(u^{n+1}_{i,j+1} - 2u^{n+1}_{i,j} + u^{n+1}_{i,j-1})\right] \\
&= a_{i,j}u^n_{i,j} + \frac{1}{2}\left[ \frac{\Delta t}{\Delta x^2}(u^n_{i+1,j} - 2u^n_{i,j} + u^n_{i-1,j}) + \frac{\Delta t}{\Delta y^2}(u^n_{i,j+1} - 2u^n_{i,j} + u^n_{i,j-1})\right]  \\
& \qquad + \frac{\Delta t}{2} F_{i,j}G_{i,j}^{n+1}  + \frac{\Delta t}{2} F_{i,j}G_{i,j}^n,
\end{cases}
\end{equation*}
\iffalse
\begin{equation*}\label{dis11}
(3D)
\begin{cases}
&a_{i,j,k}\frac{u^{n+1}_{i,j,k} - u^n_{i,j,k}}{\Delta t} = \frac{1}{2}\left(\frac{u^{n+1}_{i-1,j,k} - 2u^{n+1}_{i,j,k} + u^{n+1}_{i+1,j,k}}{\Delta x^2} + \frac{u^{n+1}_{i,j-1,k} - 2u^{n+1}_{i,j,k} + u^{n+1}_{i,j+1,k}}{\Delta y^2} + \frac{u^{n+1}_{i,j,k-1} - 2u^{n+1}_{i,j,k} + u^{n+1}_{i,j,k+1}}{\Delta z^2}\right) \\
&+ \frac{1}{2}\left(\frac{u^n_{i-1,j,k} - 2u^n_{i,j,k} + u^n_{i+1,j,k}}{\Delta x^2} + \frac{u^n_{i,j-1,k} - 2u^n_{i,j,k} + u^n_{i,j+1,k}}{\Delta y^2} + \frac{u^n_{i,j,k-1} - 2u^n_{i,j,k} + u^n_{i,j,k+1}}{\Delta z^2} \right)\\
& \qquad+ \frac{1}{2}F_{i,j,k} G^{n+1}_{i,j,k} + \frac{1}{2}F_{i,j,k}G^n_{i,j,k}.
\end{cases}
\end{equation*}
\fi
and in the case of 3D, we have
\begin{equation*}\label{dis22}
\begin{cases}
&a_{i,j,k}u^{n+1}_{i,j,k} - \frac{1}{2}  \left[ \frac{\Delta t}{\Delta x^2}(u^{n+1}_{i+1,j,k} - 2u^{n+1}_{i,j,k} + u^{n+1}_{i-1,j,k}) + \frac{\Delta t}{\Delta y^2}(u^{n+1}_{i,j+1,k} - 2u^{n+1}_{i,j,k} + u^{n+1}_{i,j-1,k}) \right. \\
 & \qquad \qquad  \left.+ \frac{\Delta t}{\Delta z^2}(u^{n+1}_{i,j,k+1} - 2u^{n+1}_{i,j,k} + u^{n+1}_{i,j,k-1})\right] \\
&= a_{i,j,k}u^n_{i,j,k} + \frac{1}{2}   \left[ \frac{\Delta t}{\Delta x^2}(u^n_{i+1,j,k} - 2u^n_{i,j,k} + u^n_{i-1,j,k}) + \frac{\Delta t}{\Delta y^2}(u^n_{i,j+1,k} - 2u^n_{i,j,k} + u^n_{i,j-1,k})\right. \\
& \qquad \qquad \left. + \frac{\Delta t}{\Delta z^2}(u^n_{i,j,k+1} - 2u^n_{i,j,k} + u^n_{i,j,k-1})\right]  + \frac{\Delta t}{2} F_{i,j,k}G_{i,j,k}^{n+1}  + \frac{\Delta t}{2} F_{i,j,k}G_{i,j,k}^n,
\end{cases}
\end{equation*}
where $u^n_{i,j} := u(x_i,y_j,t_n)$, $u^n_{i,j,k} := u(x_i,y_j,z_k, t_n),$  in which $x_i = i\Delta x$ for $i = 0,1,2,··· ,N_x $, $y_j = j\Delta y$ for $j = 0,1,2,··· ,N_y $, $z_k = k\Delta z$ for $k = 0,1,2,··· ,N_z $.
We observe that the scheme \eqref{dis1} can be written in the form of system of linear equations
$A u^{n+1} = f(u^n)$ and solved for $u^{n+1}$ by known $f(u^n)$.\par
To discretize the first-order boundary condition of direct problem \eqref{prb1}, we use the central difference scheme for the discretization. This allows to obtain a numerical approximation of higher order than ordinary(backward or forward) finite difference approximation.
Similarly, FD scheme can be derived   for the solution of the adjoint problem \eqref{adj11}. In the 2D case, the FD scheme  for discretization of \eqref{adj11} can be expressed as follows:
\begin{equation*}\label{adj1}
\begin{cases}
&-a_{i,j}\frac{(\lambda^{n+1}_{i,j} - \lambda^n_{i,j})}{\Delta t} = \frac{1}{2}\left(\frac{\lambda^{n+1}_{i+1,j} - 2\lambda^{n+1}_{i,j} + \lambda^{n+1}_{i-1,j}}{\Delta x^2} + \frac{\lambda^{n+1}_{i,j+1} - 2\lambda^{n+1}_{i,j} + \lambda^{n+1}_{i,j-1}}{\Delta y^2}\right) \\
&+ \frac{1}{2}\left(\frac{\lambda^n_{i+1,j} - 2\lambda^n_{i,j} + \lambda^n_{i-1,j}}{\Delta x^2} + \frac{\lambda^n_{i,j+1} - 2\lambda^n_{i,j} + \lambda^n_{i,j-1}}{\Delta y^2} \right) \\
&- \frac{1}{2} (u^{n+1}_{i,Ny}- \tilde u^{n+1}_{i,Ny}) \delta_{\rm obs} z_\delta(x_{i},y_{j})  
- \frac{1}{2} (u^n_{i,Ny}- \tilde u^n_{i,Ny}) \delta_{\rm obs} z_\delta(x_{i},y_{j}),
\end{cases}
\end{equation*}
which can be rewritten in the form
\begin{equation}\label{adj22}
\begin{cases}
&a_{i,j}\lambda^n_{i,j} - \frac{1}{2}\left[ \frac{\Delta t}{\Delta x^2}(\lambda^n_{i+1,j} - 2\lambda^n_{i,j} + \lambda^n_{i-1,j}) + \frac{\Delta t}{\Delta y^2}(\lambda^n_{i,j+1} - 2\lambda^n_{i,j} + \lambda^n_{i,j-1})\right] \\
&= a_{i,j}\lambda^{n+1}_{i,j} + \frac{1}{2}\left[ \frac{\Delta t}{\Delta x^2}(\lambda^{n+1}_{i+1,j} - 2\lambda^{n+1}_{i,j} + \lambda^{n+1}_{i-1,j}) + \frac{\Delta t}{\Delta y^2}(\lambda^{n+1}_{i,j+1} - 2\lambda^{n+1}_{i,j} + \lambda^n_{i,j-1})\right]  \\
& \qquad- \frac{\Delta t}{2} (u^{n+1}_{i,Ny}- \tilde u^{n+1}_{i,Ny})   \delta_{\rm obs}  z_\delta(x_{i},y_{j})
- \frac{\Delta t}{2} (u^n_{i,Ny}- \tilde u^n_{i,Ny}) \delta_{\rm obs}   z_\delta(x_{i},y_{j}),
\end{cases}
\end{equation}
where the smoothing function $z_\delta(x_{i},y_{j})$ is  centered at $(x_{i},y_{j})$ on $\partial_1\Omega$.
% Now simplified form for the 3D case is given as follows:
Now rearranged form for the 3D case is given as follows:
\begin{eqnarray}\label{adj2}
\begin{cases}
&a_{i,j,k}\lambda^n_{i,j,k} - \frac{1}{2}\left[ \frac{\Delta t}{\Delta x^2}(\lambda^n_{i+1,j,k} - 2\lambda^n_{i,j,k} + \lambda^n_{i-1,j,k}) + \frac{\Delta t}{\Delta y^2}(\lambda^n_{i,j+1,k} - 2\lambda^n_{i,j,k} + \lambda^n_{i,j-1,k})  \right.\\
& \left. \qquad \qquad \qquad+  \frac{\Delta t}{\Delta z^2}(\lambda^n_{i,j,k+1} - 2\lambda^n_{i,j,k} + \lambda^n_{i,j,k-1})\right] \\
&= a_{i,j,k}\lambda^{n+1}_{i,j,k} + \frac{1}{2}\left[ \frac{\Delta t}{\Delta x^2}(\lambda^{n+1}_{i+1,j,k} - 2\lambda^{n+1}_{i,j,k} + \lambda^{n+1}_{i-1,j,k}) + \frac{\Delta t}{\Delta y^2}(\lambda^{n+1}_{i,j+1,k} - 2\lambda^{n+1}_{i,j,k} + \lambda^{n+1}_{i,j-1,k})  \right. \\
& \left. \qquad \qquad \qquad \qquad  +  \frac{\Delta t}{\Delta z^2}(\lambda^{n+1}_{i,j,k+1} - 2\lambda^{n+1}_{i,j,k} + \lambda^{n+1}_{i,j,k-1})\right]  \\
& - \frac{\Delta t}{2} (u^{n+1}_{i,j,Nz}- \tilde u^{n+1}_{i,j,Nz})   \delta_{\rm obs}  z_\delta(x_{i},y_{j},z_{k})
- \frac{\Delta t}{2} (u^n_{i,j,Nz}- \tilde u^n_{i,j,Nz}) \delta_{\rm obs}   z_\delta(x_{i},y_{j},z_{k}),
\end{cases}
\end{eqnarray}
where the function $z_\delta(x_{i},y_{j}, z_{k})$ is centered at $(x_{i},y_{j},z_{k})$ on $\partial_1\Omega$.
The scheme \eqref{adj22} and \eqref{adj2}  can be written as a system of linear equations
$B \lambda^n = g(\lambda^{n+1}) $ and solved backward in time for $\lambda^n$ using known values of $g(\lambda^{n+1})$.
The analysis of the proposed FD scheme is studied in several works(see, for example,\cite{strikwerda2004finite}).
\subsection{Optimization Algorithms}
\label{sec:algo}
In this section, we will formulate an optimization algorithm based on Conjugate Gradient Algorithm(CGA) for the computation of the optimal
 source function $F(x)$.\par
By using \eqref{gra}, we define the gradient at iteration $m$ in the optimization  algorithm (CGA)
 as
\begin{equation}\label{Bhm}
\begin{split}
  g_h^m(x) = -\int_0^T  G(x,t) \lambda_h^m  d\tau  + \gamma^m  (F_h^m - F_0)(x),
\end{split}
\end{equation}
where $F_h^m$ is the computed source function at each iteration
$m$ of the algorithm, $u_h(x,t, F_h^m)$ and $\lambda_h( x,t,F_h^m)
$ are computed solutions of the state problem (\ref{prb1}) and the
adjoint problem (\ref{adj11}), respectively, with $F:= F_h^m$,
on the finite difference mesh $K_h$. The regularization parameter
$\gamma^m$ can be computed via iterative rules of \cite{BKS} as
\begin{equation}\label{iterative}
  \gamma^m = \frac{\gamma^0}{(m+1)^p}, ~~p \in (0,1),
  \end{equation}
where $\gamma^0$ is an initial guess for the regularization parameter. 

Let us formulate the
conjugate gradient algorithm (CGA) for computing the optimal solution of the
functional (\ref{L1}), or
the source function $F$. The iterations in CGA are performed via the following iterative rule:
\begin{equation}\label{cgm}
F_h^{m+1}(x) =  F_h^m(x)  - \alpha^m d_h^m(x),
  \end{equation}
where $\alpha^m$ are iteratively updated step sizes in the gradient update and $d_h^m$ is the direction of descent, which is computed for the conjugate gradient method at iteration $m$ as
\begin{equation}\label{dm}
d_h^m = g_h^m + \beta^m d_h^{m-1},
  \end{equation}
where $g_h^m$ is the gradient given in \eqref{Bhm}. Here, the conjugate coefficient $\beta^m$ is given by the Fletcher–Reeves method(see,\cite{fletcher1964function,daniel1971approximate}) is computed as
\begin{equation*}\label{beta}
\begin{split}
 \beta^m &= \frac{\| g_h^m(x)\|^2}{\| g_h^{m-1}(x)\|^2}.
\end{split}
\end{equation*}
The step-size $\alpha^m$   in the CGA update  \eqref{cgm}  is computed such that
it minimizes the Tikhnov functional  $J_{\gamma}(u_h^m, F_h^m - \alpha^m d^m)$.
The next lemma   provides a formula for the computation of $\alpha^m$
 in CGA update \eqref{cgm}.
\begin{Lemma}\label{step1}
Suppose the iterations in the  CGA are given by \eqref{cgm}.
Then step size  $\alpha^m$  at iteration $m$ of the conjugate gradient update \eqref{cgm}
 can be computed as
\begin{eqnarray}\label{alphacgm1}
\begin{split}
&\alpha^m = \frac{\int_\Omega\int_0^T( u(x,t;F^m_h) -  \widetilde u(x,t))   \delta u(x,t;F^m_h) \delta_{\rm obs} z_\delta (x) ~dxdt + \gamma^m \int_\Omega(F^m_h - F_0)(x) d^m_h(x)~~ dx}{\gamma^m \int_\Omega(d^m_h(x))^2~~dx + \int_\Omega\int_0^T(\delta u(x,t;F^m_h))^2 \delta_{\rm obs} z_\delta (x)~~dxdt},
\end{split}
\end{eqnarray}
where $\delta u$ is the solution of the sensitivity equation \eqref{grad1}.
\end{Lemma}
\begin{proof}
The search step size $\alpha^m$ is computed by the exact line search method, which is given by the minimization of Tikhonov functional, i.e.,
\begin{equation}\label{min1}
 \frac{\partial}{\partial \alpha^m}J_\gamma(u,F^m_h - \alpha^md^m_h) = 0,  
\end{equation}
where, 
\begin{equation}\label{para1}
\begin{split}
 J_\gamma(u,F^m_h - \alpha^md^m_h) &=  \frac{1}{2} \int_\Omega\int_0^T( u(x,t;F^m_h - \alpha^md^m_h) -  \widetilde u(x,t))^2  \delta_{\rm obs} z_\delta (x) ~dxdt  \\
&  \ \ \ \ +\frac{1}{2} \gamma^m \int_\Omega((F^m_h - \alpha^md^m_h) -  F_0)^2(x)~~ dx.   
\end{split}
\end{equation}
Setting $\partial F^m_h = d^m_h$, and linearize the term $u(x,t;F^m_h - \alpha^m d^m_h)$ by Taylor's series expansion, we have 
\begin{equation}\label{min2}
 u(x,t;F^m_h - \alpha^m d^m_h) \approx   u(x,t;F^m_h) - \alpha^m \delta u (x,t;F^m_h).  
\end{equation}
Using \eqref{min2} in \eqref{para1}, we obtain  from \eqref{min1} that 
\begin{align*}
&\int_\Omega\int_0^T( u(x,t;F^m_h) -  \widetilde u(x,t))  \delta u(x,t;F^m_h) \delta_{\rm obs} z_\delta(x)  ~dxdt + \gamma^m \int_\Omega(F^m_h - F_0)(x) d^m_h(x)~~ dx \\
& - \alpha^m\int_\Omega\int_0^T(\delta u(x,t;F^m_h))^2 \delta_{\rm obs} z_\delta (x)~~dxdt  - \alpha^m\gamma^m \int_\Omega(d^m_h(x))^2~~dx=0. 
\end{align*}
This leads to the definition of step size given in \eqref{alphacgm1}. \qed \\
\end{proof}
\noindent We summarize all steps of the CGA algorithm in Algorithm 1.
\begin{algorithm}[hbt!]
  \centering
  \caption{ Conjugate Gradient Algorithm (CGA)}
    \begin{algorithmic}[1]
      \STATE Initialization:
\begin{itemize}
     \item  Choose 
       the  finite difference mesh $K_h$ in $\Omega$  and discretization  $J_\tau$ of the time interval $[0,T] .$
\item Choose  the initial approximation for the source function
  $F_h^0= F^0$  at  $K_h$.
  \item Choose the initial value of the regularization parameter   $\gamma^0$.
\end{itemize}
Compute the sequence of source functions
  $F_h^m$ via the following steps:
\STATE  Compute the solutions $u_h(x,t, F_h^m) $ and
  $\lambda_h(x,t, F_h^m) $ of the state  (\ref{prb1})
  and adjoint  (\ref{adj11})
problems, respectively, on $K_h$.
\STATE  Compute regularization parameter  $\gamma^m$  as in \eqref{iterative}.
\STATE Compute the direction of descent $d_h^m$ as in \eqref{dm}.
\STATE Compute the solution $U$ of the problem \eqref{sens2} by setting $\hat F = d_h^m.$
\STATE Compute the search step size $\alpha^m$ from \eqref{alphacgm1} by setting $\delta u = U.$
\STATE   Compute new values of the source function
  $F_h:= F_h^{m+1}$   using   conjugate gradient update
\begin{equation}\label{cgm1}
\begin{split}
F_h^{m+1} &=  F_h^m  - \alpha^m d_h^m,
\end{split}
\end{equation}
where $d_h^m$ is computed via \eqref{dm}.
Here, $d^0(x)= g^0(x)$. In (\ref{cgm1}) the step size $\alpha^m$ in
the conjugate gradient update is computed via \eqref{alphacgm1}.
\STATE Stop the algorithm and obtain the function $F_h$ at the iteration
$\bar m = m$ if any one of the following  criteria is satisfied:
$\|g^m\|_{L^2(\Omega)} \leq \theta_1$, or
$\|e^m\|_{L^2(\Omega)}  \leq \theta_2 $,
 where the relative error $e^m$ is computed as
\begin{equation*}\label{An}
  e^m: =  \| F_h^m - F_h^{m-1}\|/  \| F_h^m  \|.
  \end{equation*}
 Here, $\theta_i,i=1,2$
 are the
tolerances. Otherwise set
$m:=m+1$ and go to step 2.
 \end{algorithmic}
\end{algorithm}
\subsection{Numerical Examples}
\label{sec:numex}
This section describes numerical examples of the reconstruction of the source function $F$ in ISP in 2D and 3D.
The goal of our numerical tests is to reconstruct the source function $F$ in the
domain $\Omega$,
which we set as
 \begin{equation*}
 \Omega = \left\{ x= (x_1,x_2): x_1 \in [0 , 1], x_2 \in [0, 1] \right\} \,\, \mbox{in 2D}.
 \end{equation*}
and,  \begin{equation*}
 \Omega = \left\{ x= (x_1,x_2, x_3): x_1 \in [0 , 1], x_2 \in [0, 1], x_3 \in [0,1] \right\} \,\, \mbox{in 3D}.
 \end{equation*}

To determine the source $F$  in ISP of the model problem \eqref{prb1}, we minimize the Tikhonov functional
\eqref{L1}  where  the function $\tilde{u}(x,t)$ is the measured simulated function at the observation points placed at the part of the boundary
  $\partial_1 \Omega$. Here, the function $u(x,t)$ corresponds to
the simulated solution of the model problem with the exact source function. We assume in our computations that we don't know the source function and we are working only with the measured function  $\tilde u(x,t)$.\par
For the solution of the optimization problem  $\displaystyle \min_{F\in \mathcal F} ~ J_\gamma(u,F),$ we use  CGA algorithm of section \ref{sec:algo} with an iterative choice of the regularization
parameter $\gamma$ in \eqref{L1}, which is computed by \eqref{iterative}.
The iterative update of the regularization parameter $\gamma^m$ \eqref{iterative}  was
proposed  and justified in \cite{BKS}
for the solution of the inverse problem using  
gradient-like methods. Here 
 $m$ is the number
of iterations in the  CGA, $p \in (0,1)$ and
$\gamma^0$ is an initial guess for the regularization parameter.
As an initial guess $\gamma^0$, one can take 
$\gamma^0 =\delta^\zeta$, where $\delta$ is the known noise level  in the data,  and
$\zeta$ is a small number in the interval $(0,1)$, see  the explanation
 for such a choice in \cite{Klibanov_Bakushinsky_Beilina}.
 
In order to check the performance of the reconstruction algorithm, we supply the simulated noisy
data at $\partial_1 \Omega$ by adding the normally distributed 
Gaussian noise with mean $\mu = 0$ to the simulated
data at the boundary $\partial_1 \Omega$. Then,
we smooth out this data to get reasonable
reconstructions.
\subsection{Reconstruction of the Source Function $F$ in 2D} In this section, we consider a set of 2D examples for the reconstruction of the source term with different initial guess $F_0.$  \\
\begin{figure}
\begin{center}  
\begin{tabular}{cc}
 \subfloat[] {\includegraphics[scale=0.45, clip = true, trim = 0.0cm 0.0cm 0.0cm 0.0cm ]{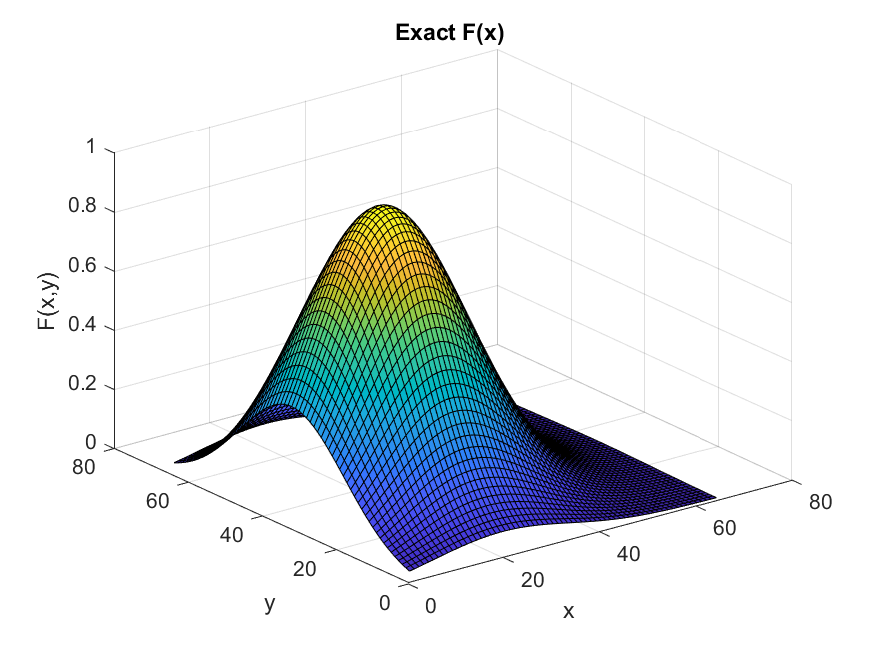}  \label{fig:Test1A}} &
 \subfloat[] {\includegraphics[scale=0.45, clip = true, trim = 0.0cm 0.0cm 0.0cm 0.0cm]{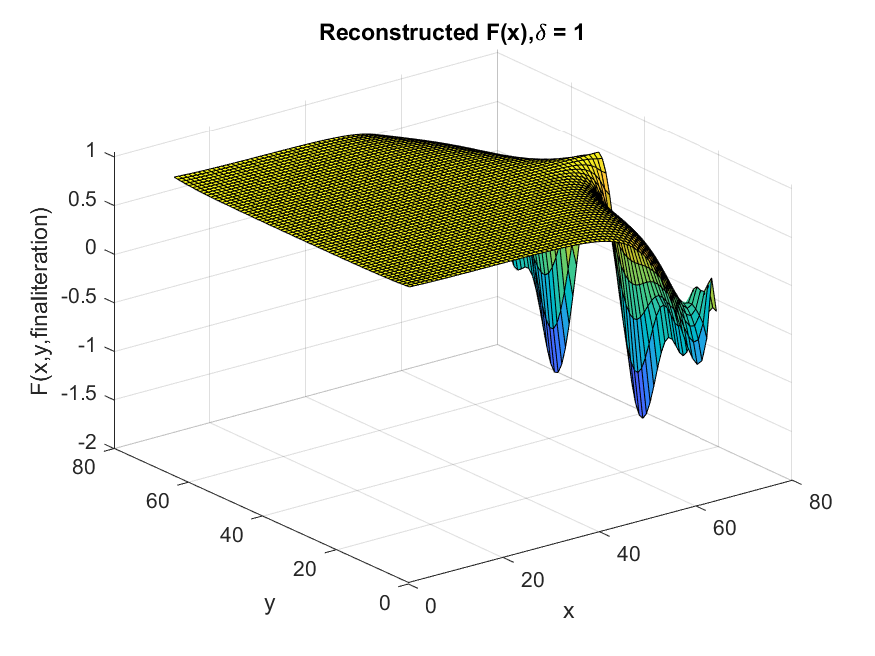} \label{fig:Test1B}}\\
 \subfloat[]{\includegraphics[scale=0.45, clip = true, trim = 0.0cm 0.0cm 0.0cm 0.0cm ]{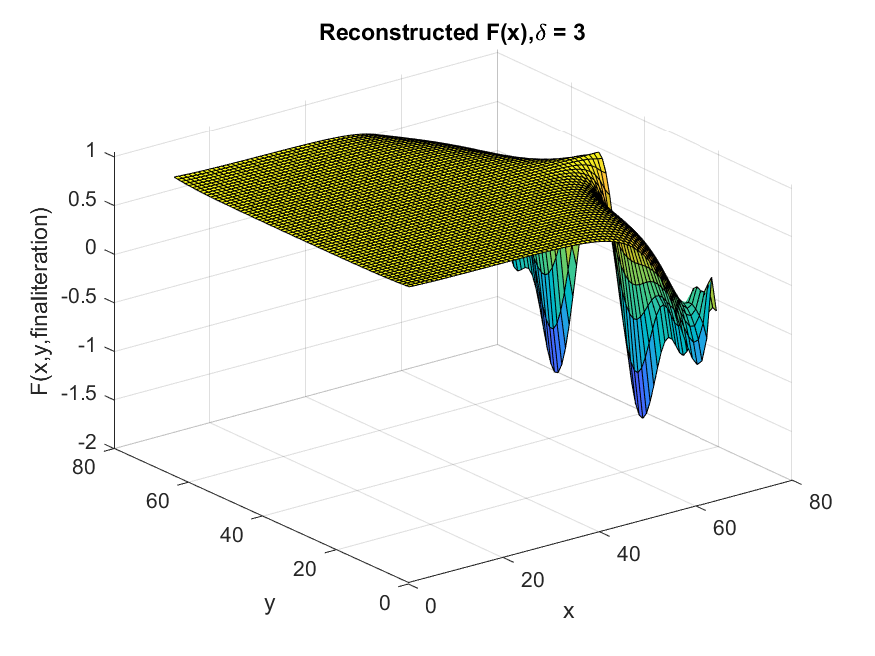}  \label{fig:Test1C}}&
 \subfloat[] {\includegraphics[scale=0.45, clip = true, trim = 0.0cm 0.0cm 0.0cm 0.0cm]{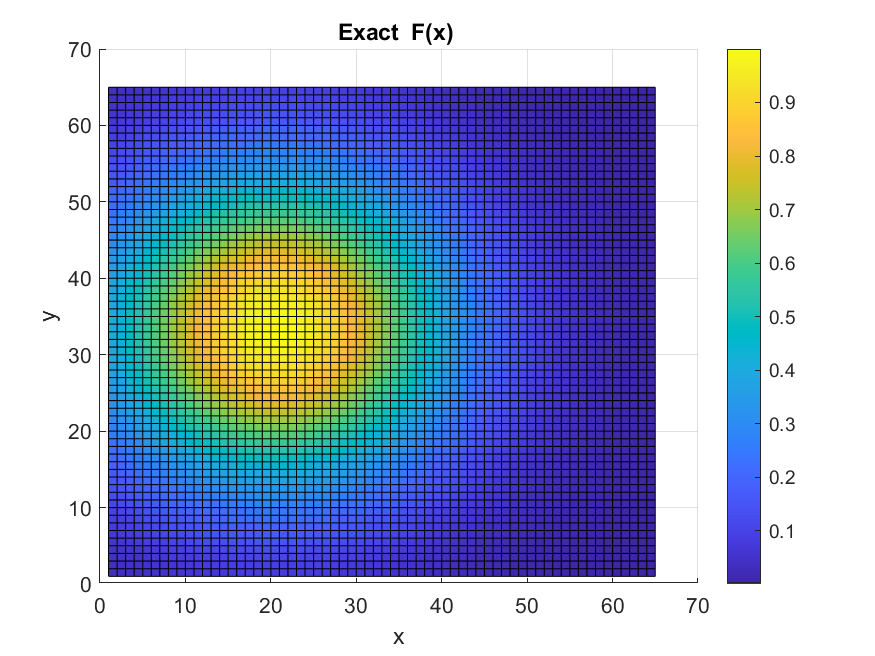} \label{fig:Test1D}}\\
  \subfloat[] {\includegraphics[scale=0.45, clip = true, trim = 0.0cm 0.0cm 0.0cm 0.0cm]{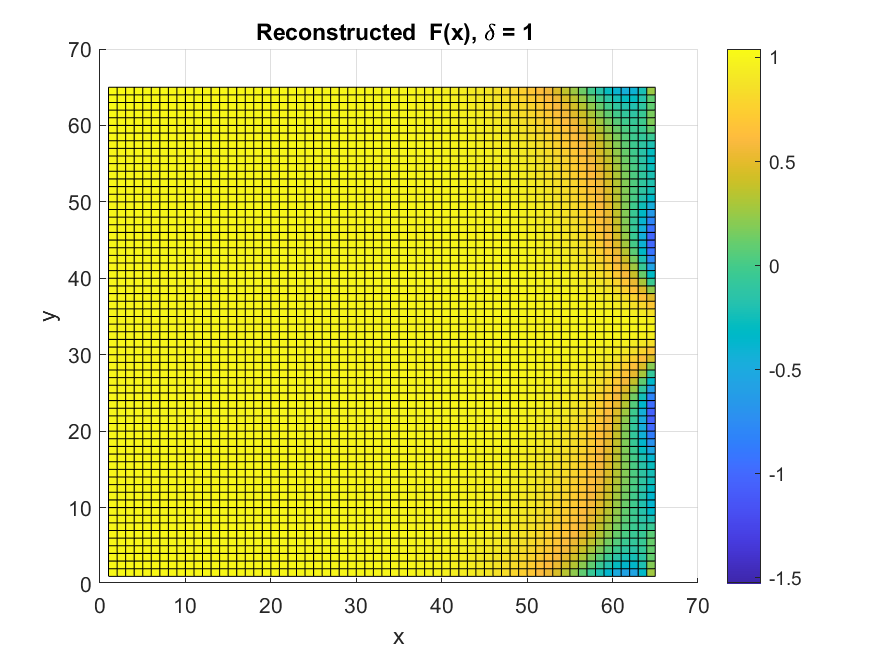} \label{fig:Test1E}}&
\subfloat[]  {\includegraphics[scale=0.45, clip = true, trim = 0.0cm 0.0cm 0.0cm 0.0cm]{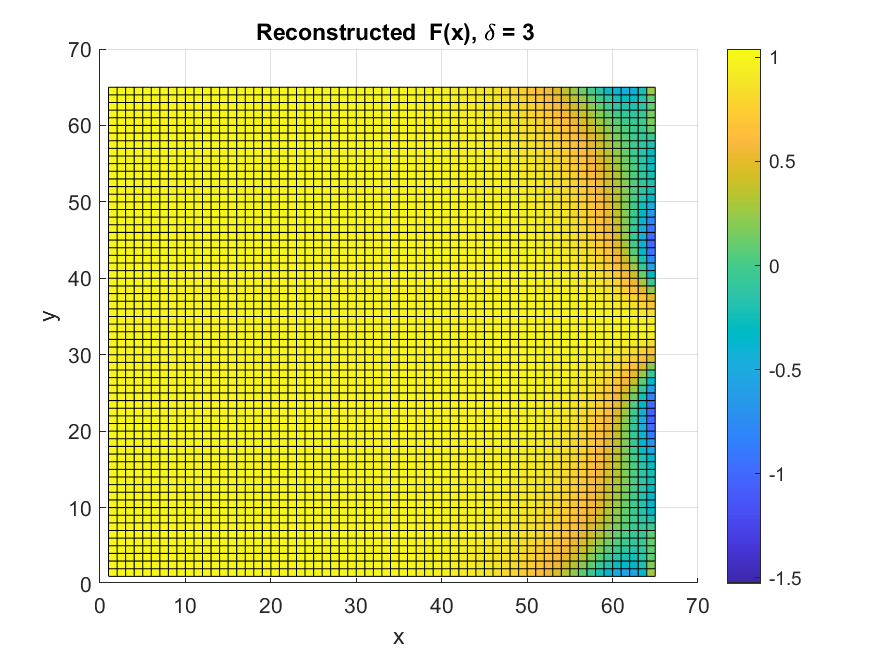} \label{fig:Test1F}}\\
\subfloat[]{\includegraphics[scale=0.45, clip = true, trim = 0.0cm 0.0cm 0.0cm 0.0cm]{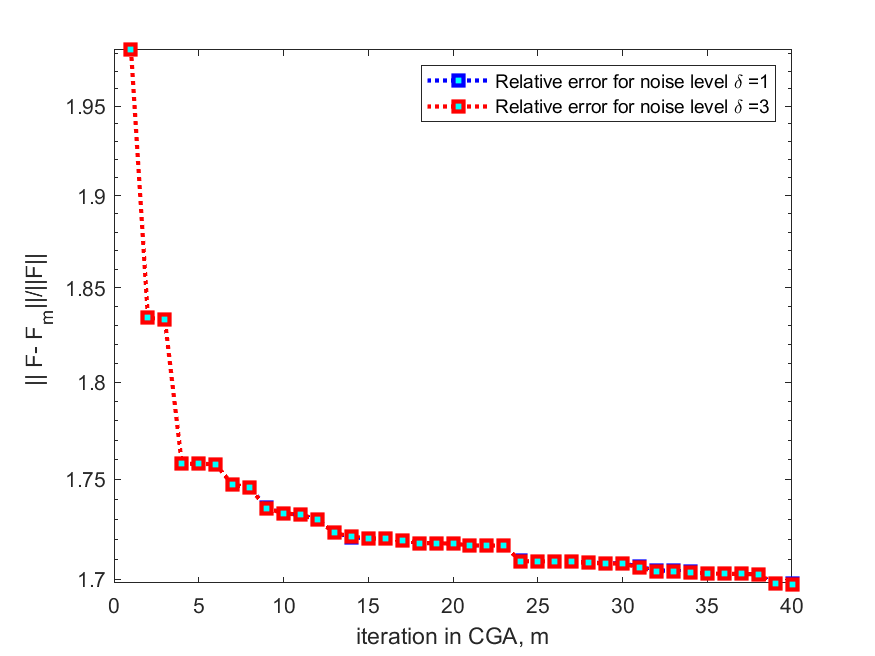} \label{fig:Test1G}} &
\subfloat[]{\includegraphics[scale=0.45, clip = true, trim = 0.0cm 0.0cm 0.0cm 0.0cm]{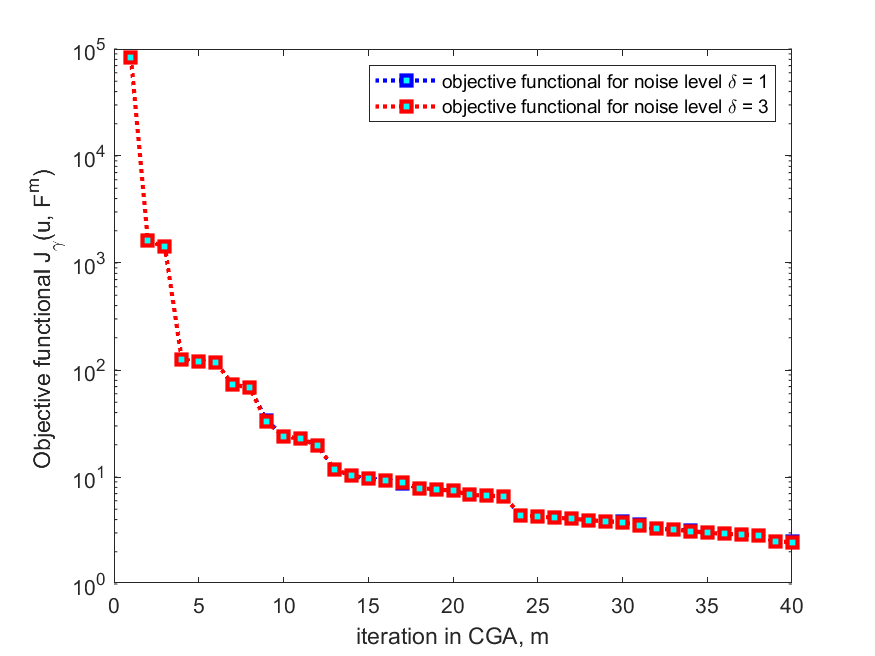} \label{fig:Test1H}} \\ 
\end{tabular}
\end{center}
\caption{Exact $F$ and reconstructed function $F^m$ at the iteration $m= 40$ of CGA. Computations are performed on the mesh with $h = 2^{-6}$.}
 \label{fig:Test1ab}
 \end{figure}
\begin{figure}
\begin{center}  
\begin{tabular}{cc}
 \subfloat[]{\includegraphics[scale=0.46, clip = true, trim = 0.0cm 0.0cm 0.0cm 0.0cm ]{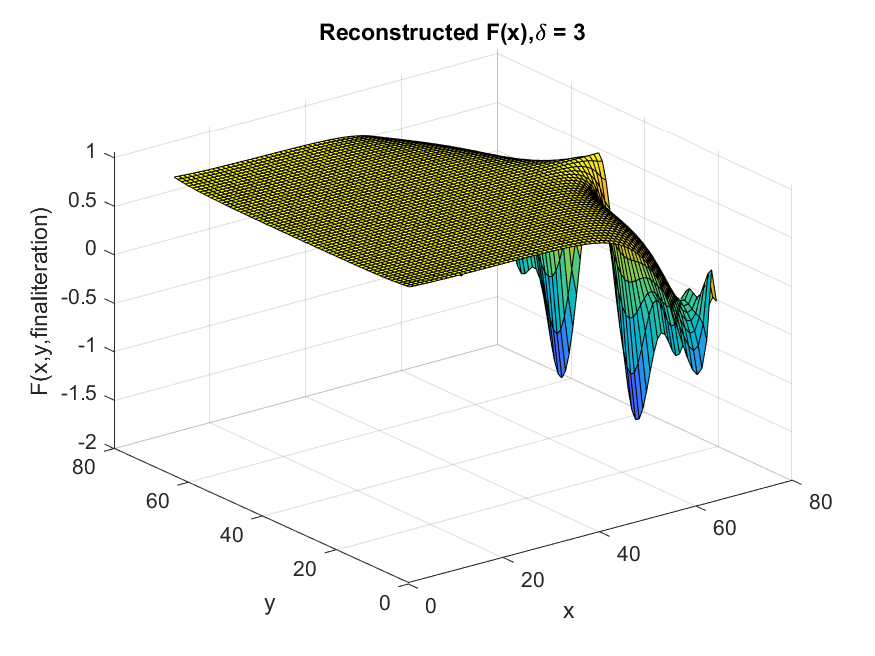} \label{fig:Test2A}} &
 \subfloat[] {\includegraphics[scale=0.46, clip = true, trim = 0.0cm 0.0cm 0.0cm 0.0cm]{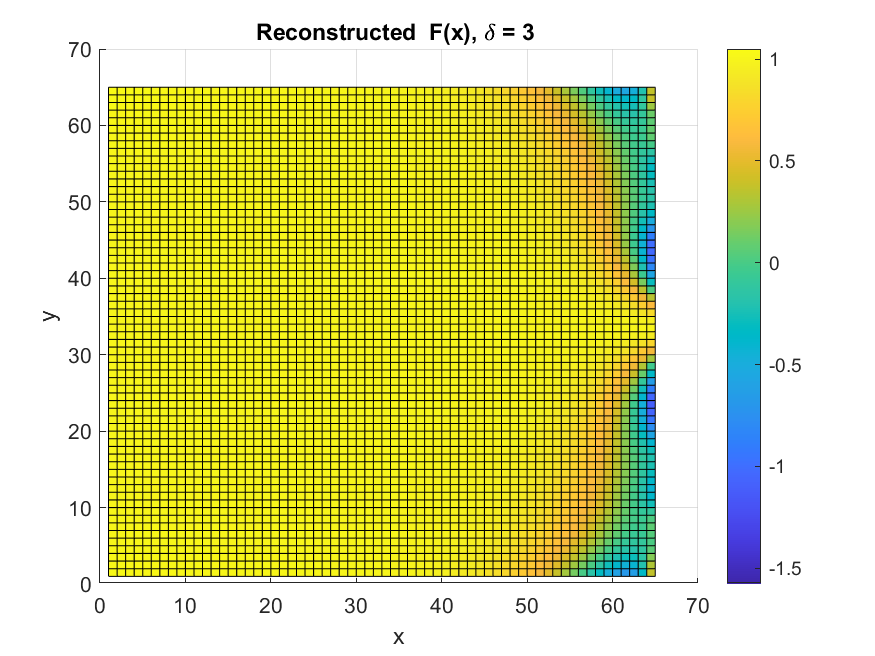}  \label{fig:Test2B}} \\
\subfloat[]{\includegraphics[scale=0.46, clip = true, trim = 0.0cm 0.0cm 0.0cm 0.0cm]{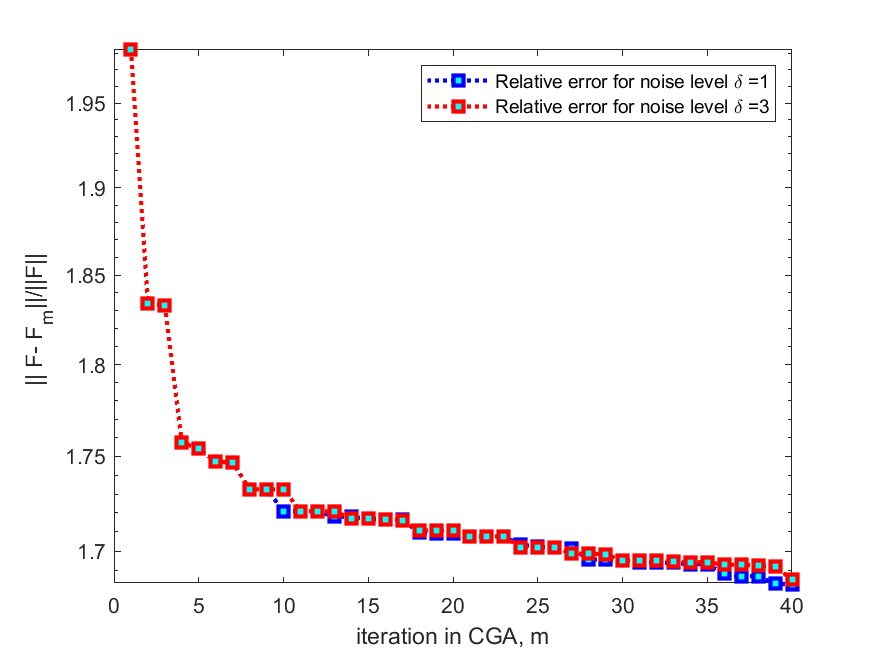} \label{fig:Test2C}} &
\subfloat[]{\includegraphics[scale=0.46, clip = true, trim = 0.0cm 0.0cm 0.0cm 0.0cm]{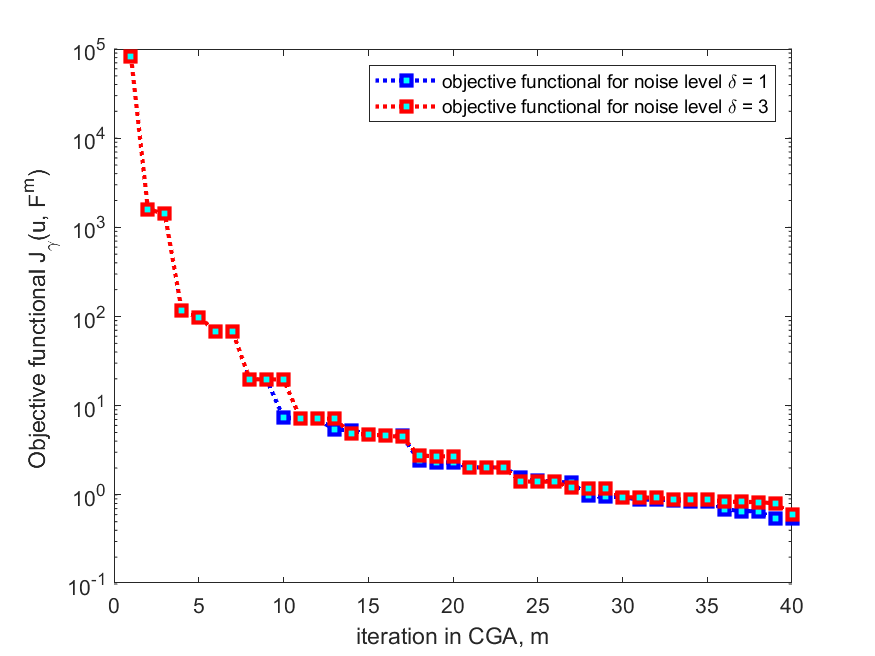} \label{fig:Test2D}} \\ 
\end{tabular}
\end{center}
\caption{Exact $F$ and reconstructed function $F^m$ at the iteration $m= 40$ of CGA. Computations are performed on the mesh with $h = 2^{-6}$.}
 \label{fig:Test2ab}
 \end{figure}
 
\noindent\textbf{Experiment 1:} \emph{(Reconstruction of $F$ using simulated data $u$ on $\partial_{1}\Omega$ and homogeneous initial guess).}
In this test, we take  $a=1 + xy$ in \eqref{prb1}. The exact source function which we want to reconstruct is given by the Gaussian function
\begin{equation}\label{exsource1}
 F(x,y) = e^{-\frac{(x-0.5)^2 + (y-0.3)^2}{m_1}}
\end{equation}
where the constant  $m_1 = const. \in (0,1)$ plays the role of shrinking of the Gaussian function. In this test, we take $m_1=0.1$
and the known function  $G(x,y,t)$  as
\begin{equation}\label{G1}
 G(x,y,t) =  (1 + xy + 2\pi^2 t) \cos(\pi x) \cos(\pi y) e^{\frac{(x-0.5)^2 + (y-0.3)^2}{m_1}}.
\end{equation}
We produce  simulated noisy data $\tilde u$ by  solving the forward problem \eqref{prb1}
 with known functions \eqref{exsource1} and \eqref{G1}
 in the domain $\Omega$ and in time $t \in [0,1]$. We add then randomly distributed
   Gaussian noise to the simulated data $u$ using Matlab's command
    \emph{normrnd} with $\delta=1$ and $\delta=3$ as
\begin{equation}\label{noise}
\tilde u = u  + normrnd(0,\delta/100,Nx,Ny)
\end{equation}
 and run CGA algorithm with smoothed noisy data
    $\tilde u$. Here, $\delta \in (0,100)$ is the noise level in percentage and $Nx, Ny$
    are the number of points in $x$ and $y$ directions of the domain $\Omega$, respectively.\par
Next, we solve ISP by starting
 CGA with a homogeneous initial guess for the source function $F^0(x,y) = 1$. Figure \ref{fig:Test1ab} presents the reconstruction results in CGA. As seen in Figures \ref{fig:Test1B}, \ref{fig:Test1C}, \ref{fig:Test1E}, and \ref{fig:Test1F}, location of the source function is correctly reconstructed, but it appears too wide. While the maximum contrast has been achieved, the reconstruction should be narrowed. The relative errors for the reconstructed source $F(x,y)$ are shown in Figure \ref{fig:Test1G}. Relative errors with varying mesh sizes for $\delta=1$ and $\delta =3$ are provided in Table \ref{testm1} by the notations $\Theta_1^{s(1)}$ and $\Theta_1^{s(3)}$.
\begin{table}[h!] 
\center
\begin{tabular}{ | l | l  |  l | l| l|l|l|l |l|l| }
\hline
$N_x =N_y$ & $l$   &   $\Theta_1^{s(1)} $ &   $\Theta_1^{s(3)} $  &   $\Theta_2^{e(1)} $  &   $\Theta_2^{e(3)} $ &  $\Theta_3^{s(1)} $ &  $\Theta_3^{s(3)} $    \\
\hline 
4  & $2$ &  2.1308   &  2.1337   &  2.1101  &   2.1121 &  0.4139 & 0.4152  \\
8  & $3$ &  2.0072  & 2.0048  & 1.9993  &    2.0039  &  0.4331  &  0.4331   \\ 
16 & $4$ &   1.8315   &  1.8317  &  1.8278  & 1.8283  &  0.3467  &  0.3487   \\
32 &  $5$ &  1.7323 &  1.7323  &  1.7287  &  1.7287 & 0.3017  &  0.3026  \\
64 &  $6$  &  1.6979 &  1.6977 &   1.6830 &   1.6854  &  0.2989  &   0.2982  \\
\hline
\end{tabular}
\caption {Relative errors for the smooth source function, when the simulated data is obtained by exact solution and computed solution are denoted as $\Theta_i^{e(\delta)}$ and $\Theta_i^{s(\delta)}$ respectively, where $\delta = 1,3$ denote the noise level. Moreover, $\Theta_i^{s(\delta)} =\Theta_i^{e(\delta)}  = \frac{\|F - F_{i,m} \|}{\| F\|} $  for mesh sizes $h_l= 2^{-l}, l=2,...,6$ in the experiment $i, i=1,2,3$, at the final optimization iteration $m=40$ in CGA.}
\label{testm1}
\end{table} \medskip \\
\noindent \textbf{Experiment 2:} \emph{(Reconstruction of $F$ using the exact solution $u$ and homogeneous initial guess).}
In this test, the exact source function $F(x,y)$ is given by \eqref{exsource1}
and the known function  $G(x,y,t)$  is given by \eqref{G1}. 
In contrast to Experiment 1, in this experiment we produce simulated noisy data $\tilde{u}$ by the exact solution $u(x,y,t) = t\cos(\pi x) \cos(\pi y) $ in place of the computed solution of the direct problem \eqref{prb1}, with $\delta=1 $ and $\delta=3 $ in \eqref{noise}.  Next, we solve ISP by starting
  CGA with an initial guess for the source function $F^0(x,y) = 1$ and $m_1=0.1$
  in \eqref{exsource1}.
Figure 
 \ref{fig:Test2ab} shows the results of reconstruction in CGA. Using this figure, we observe similar results of reconstruction as in Experiment 1, which is done using computed data, and errors are presented in Figures \ref{fig:Test2C} and \ref{fig:Test2D}. Relative error with different mesh sizes for $\delta=1 $ and $\delta=3 $ is given in Table \ref{testm1} with the notation $\Theta_2^{e(1)}$ and $\Theta_2^{e(3)}$.
\begin{figure}
\begin{center}  
\begin{tabular}{cc}
  \subfloat[]{\includegraphics[scale=0.46, clip = true, trim = 0.0cm 0.0cm 0.0cm 0.0cm ]{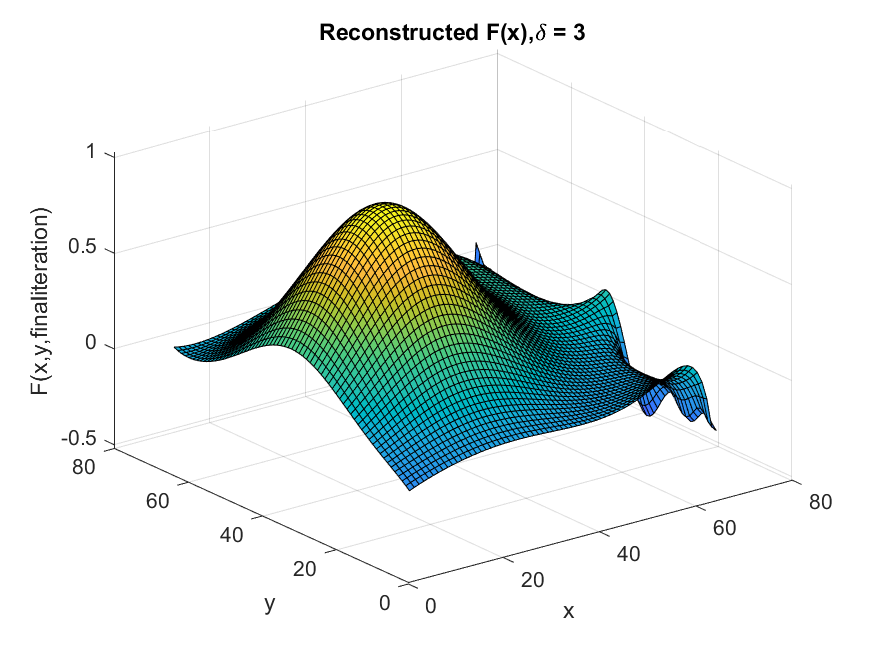} \label{fig:Test3A}}&
  \subfloat[] {\includegraphics[scale=0.46, clip = true, trim = 0.0cm 0.0cm 0.0cm 0.0cm]{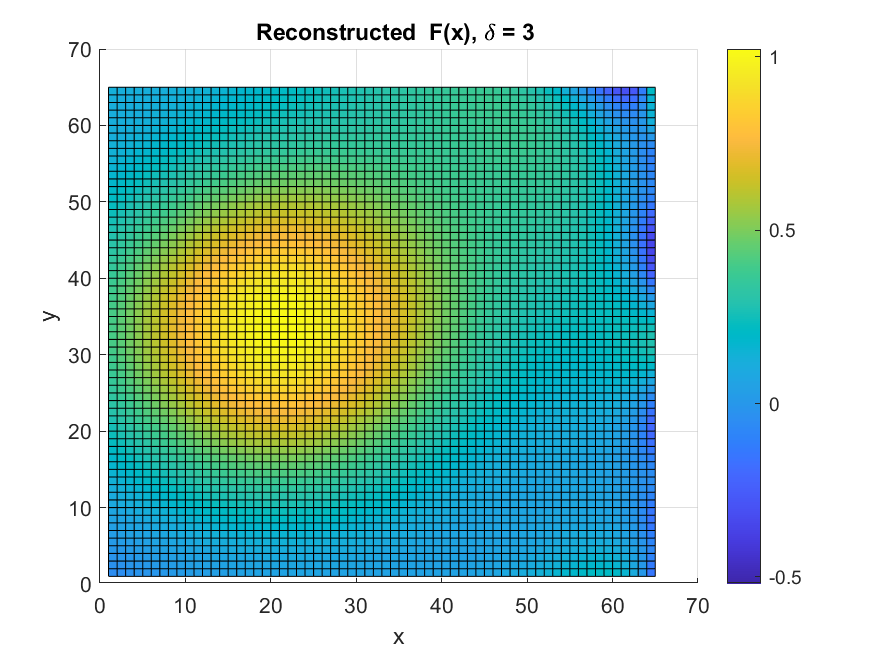}\label{fig:Test3B}}\\
\subfloat[]{\includegraphics[scale=0.46, clip = true, trim = 0.0cm 0.0cm 0.0cm 0.0cm]{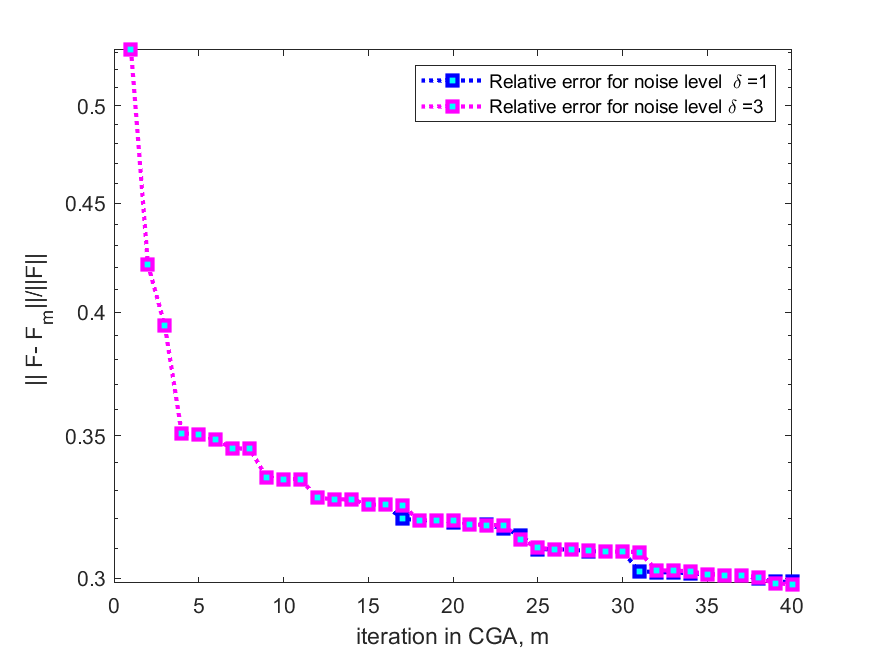} \label{fig:Test3C}}&
 \subfloat[]{\includegraphics[scale=0.46, clip = true, trim = 0.0cm 0.0cm 0.0cm 0.0cm]{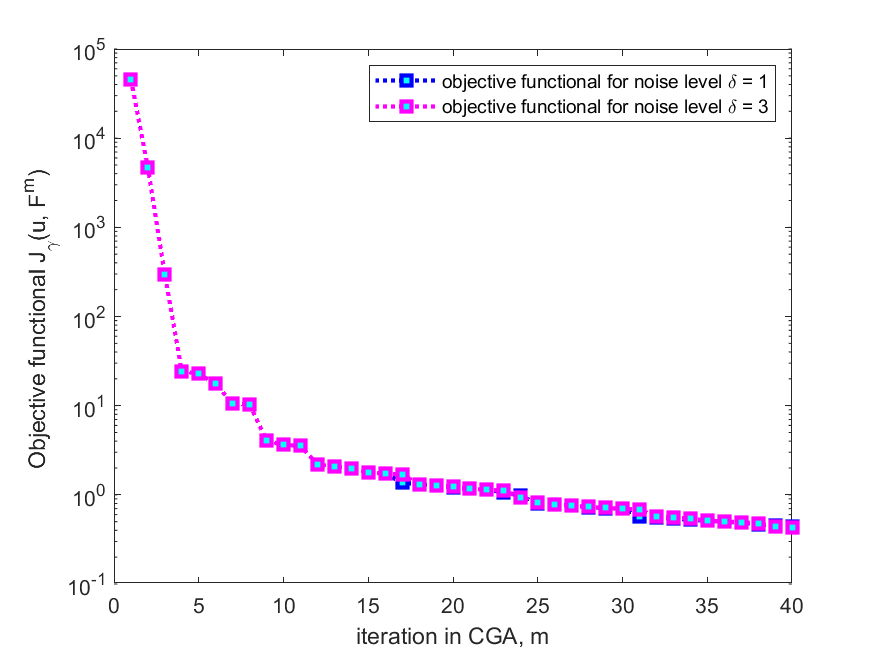} \label{fig:Test3D}} \\ 
\end{tabular}
\end{center}
\caption{Exact $F$ and reconstructed function $F^m$ at the iteration $m= 40$ of CGA. Computations are performed on the mesh with $h = 2^{-6}$.}
 \label{fig:Test3ab}
 \end{figure} \medskip\\
\noindent \textbf{Experiment 3:} \emph{(Reconstruction of $F$ using simulated data and an initial guess $F^0$ close to exact $F$).}
In this test we take  $a$, $F$ and $G$ the same as given in Experiment 1. Using the same process as in Experiment 1, we generate simulated noisy data $\tilde u$. Next, we solve ISP by initiating 
 CGA with an initial guess $F^0(x,y)= e^-\frac{(x-0.5)^2 + (y-0.3)^2}{m_1} + x^2y^2$. From Figures \ref{fig:Test3A} and \ref{fig:Test3B}, we observe that the source function's location is accurately reconstructed, and the maximum contrast is achieved. The relative errors for noise levels $\delta=1$ and $\delta=3$ are shown in Figure \ref{fig:Test3C} and relative errors for different mesh sizes, at $\delta=1$ and $\delta=3$, are presented in Table \ref{testm1} by the notation $\Theta_3^{s(1)}$ and, $\Theta_3^{s(3)}$ respectively. It is evident that the relative error in this case is sufficiently small compared to experiments 1 and 2.
 \begin{Remark} It is evident from Experiments 1, 2, and 3 that when the initial guess $F^0$ is homogeneous, the reconstruction of the source is scattered irrespective of whether the measured data is a simulated one or exact data. On the other hand, we achieve an accurate reconstruction when we have an initial guess that is close to the exact source function, even in the case of simulated measured data.
 \end{Remark}
\begin{figure}
\begin{center}  
\begin{tabular}{cc}
 \subfloat[] {\includegraphics[scale=0.455, clip = true, trim = 0.0cm 0.0cm 0.0cm 0.0cm ]{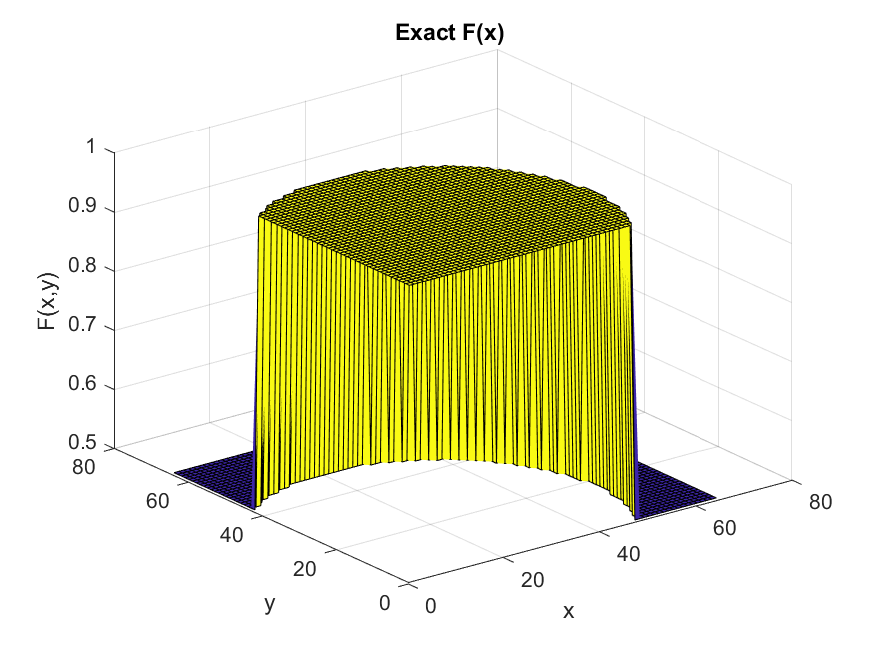}  \label{fig:Test4A}} &
 \subfloat[] {\includegraphics[scale=0.455, clip = true, trim = 0.0cm 0.0cm 0.0cm 0.0cm]{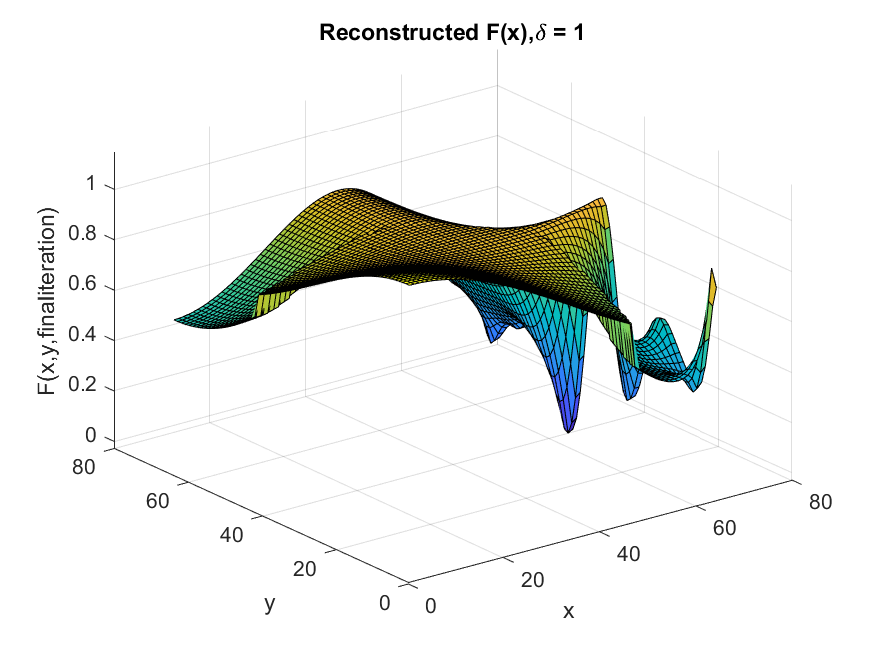}  \label{fig:Test4B}}\\
\subfloat[] {\includegraphics[scale=0.455, clip = true, trim = 0.0cm 0.0cm 0.0cm 0.0cm ]{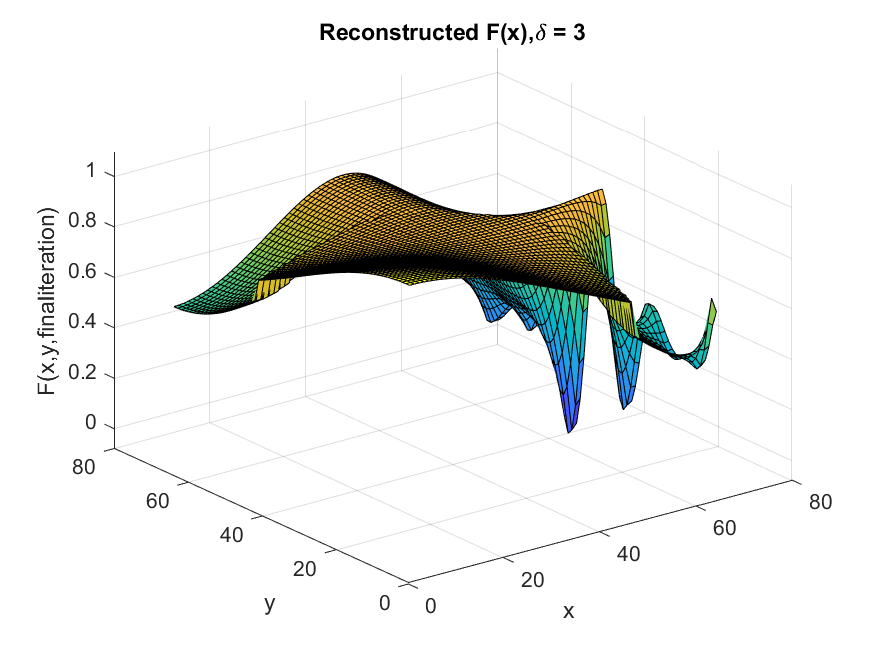}  \label{fig:Test4C}}&
\subfloat[] {\includegraphics[scale=0.455, clip = true, trim = 0.0cm 0.0cm 0.0cm 0.0cm ]{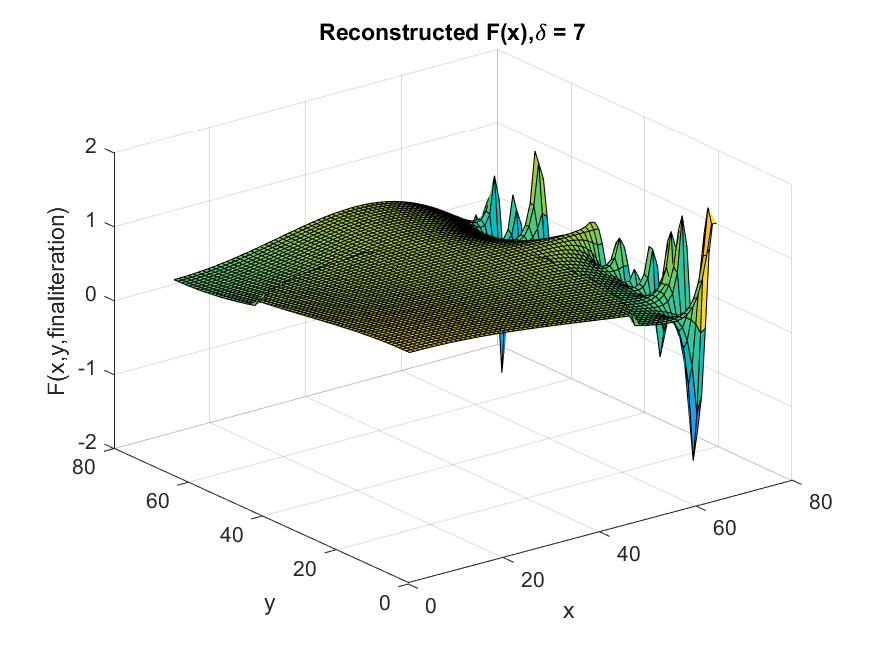}  \label{fig:Test4D}}\\
\subfloat[]  {\includegraphics[scale=0.455, clip = true, trim = 0.0cm 0.0cm 0.0cm 0.0cm]{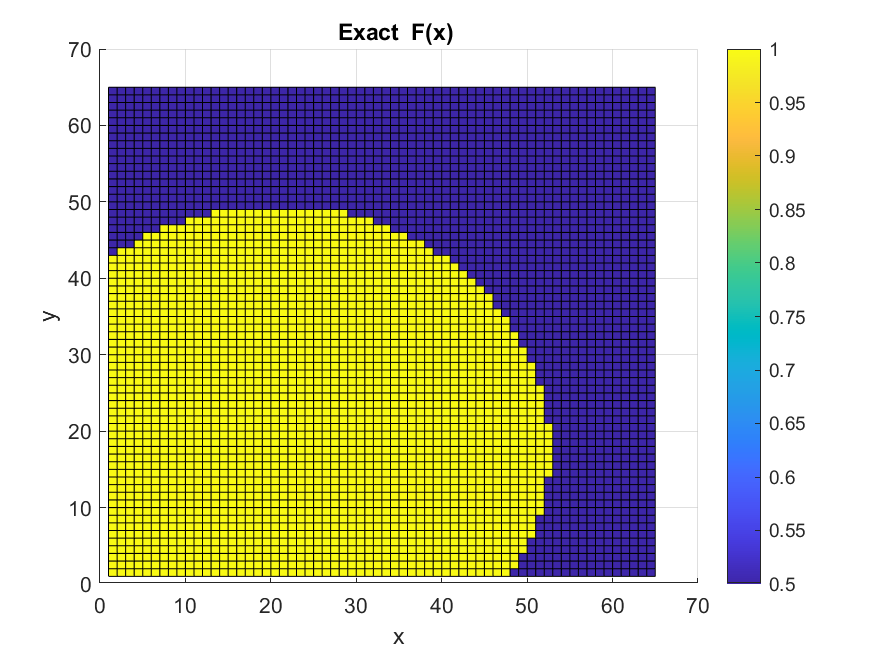}  \label{fig:Test4E}}&
\subfloat[]   {\includegraphics[scale=0.455, clip = true, trim = 0.0cm 0.0cm 0.0cm 0.0cm]{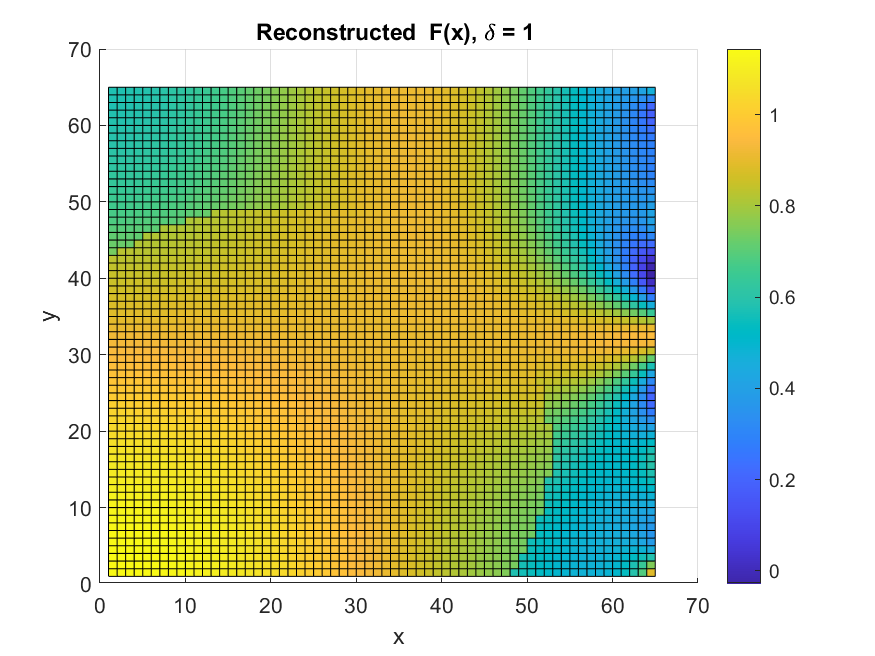}  \label{fig:Test4F}} \\
\end{tabular}
\end{center}
\end{figure}
\newpage
\setcounter{subfigure}{6}
\begin{figure}
\begin{center}  
\begin{tabular}{cc}
\subfloat[]  {\includegraphics[scale=0.455, clip = true, trim = 0.0cm 0.0cm 0.0cm 0.0cm]{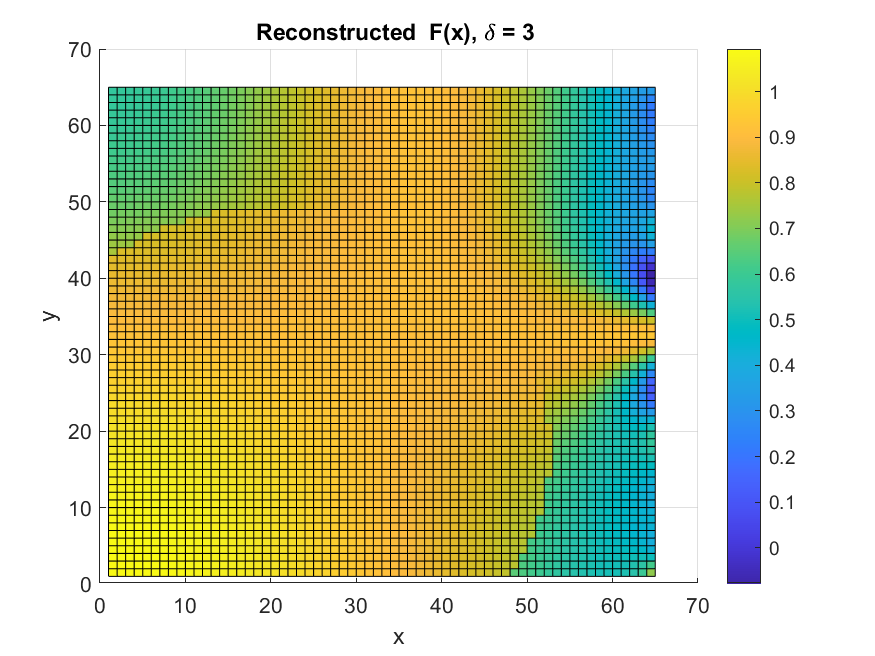}  \label{fig:Test4G}} &
\subfloat[]  {\includegraphics[scale=0.455, clip = true, trim = 0.0cm 0.0cm 0.0cm 0.0cm]{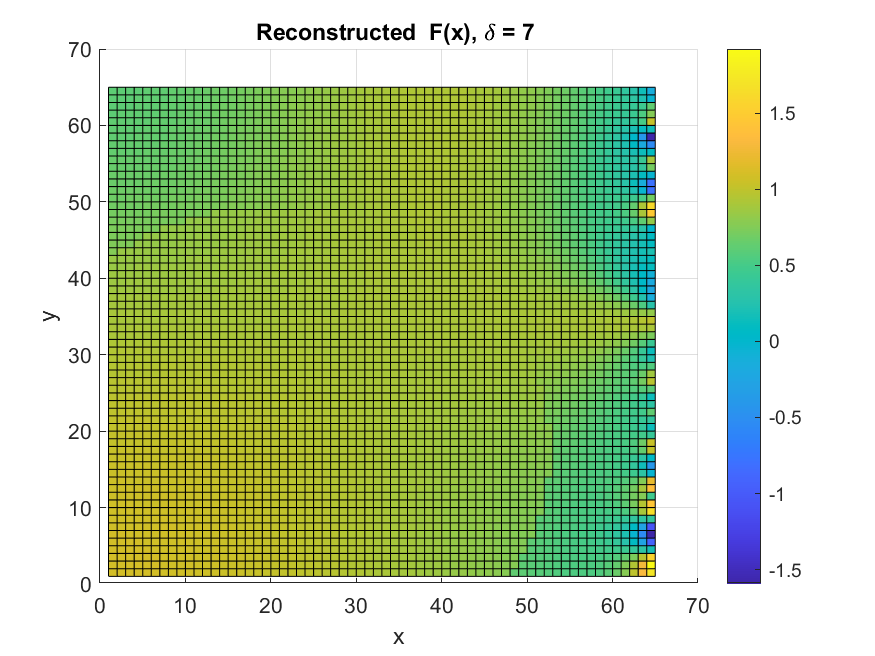}  \label{fig:Test4H}} \\
\subfloat[]{\includegraphics[scale=0.455, clip = true, trim = 0.0cm 0.0cm 0.0cm 0.0cm]{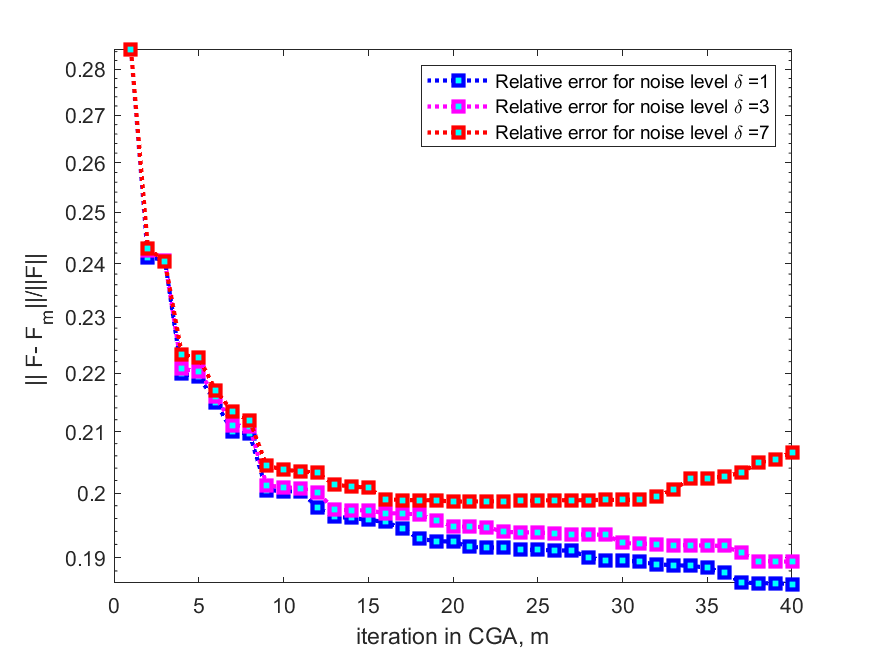}  \label{fig:Test4I}} &
\subfloat[]{\includegraphics[scale=0.455, clip = true, trim = 0.0cm 0.0cm 0.0cm 0.0cm]{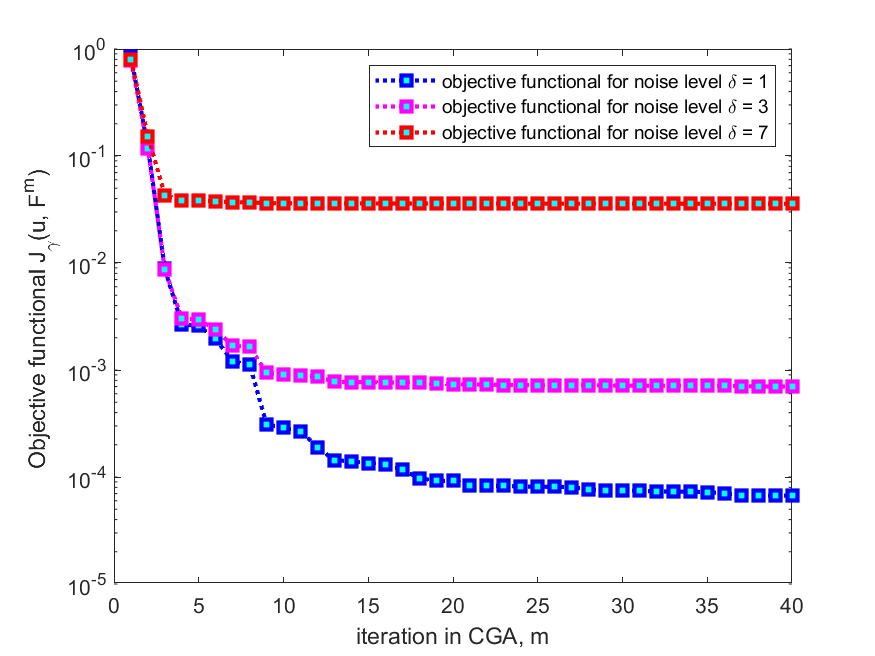} \label{fig:Test4J}} \\ 
\end{tabular}
\end{center}
\caption{Exact $F$ and reconstructed function $F^m$ at the iteration $m= 40$ of CGA. Computations are performed on the mesh with $h = 2^{-6}$.}
 \label{fig:Test4ab}
 \end{figure}
\noindent\textbf{Experiment 4:} \emph{(Reconstruction of discontinuous source function and homogeneous initial guess).}
In this test we take  $a=1 + xy$ in \eqref{prb1}. The exact source function, which we want to reconstruct, is given by the discontinuous function
\begin{equation*}\label{21gaussian3}
\begin{split}
F(x,y) = \begin{cases}
   1 & \text{if } \,\,(x-0.25)^2 + (y-0.3)^2 \leq 0.25; \\
    0.5 & \text{ otherwise}.
\end{cases}
\end{split}
\end{equation*}
The known function $G(x,y,t)$ is given by
\begin{equation*}\label{21gaussian4}
\begin{split}
G(x,y,t) = \begin{cases}
    (1 + xy + 2\pi^2 t) \cos(\pi x) \cos(\pi y) & \text{if } \,\,(x-0.25)^2 + (y-0.3)^2 \leq 0.25; \\
    2(1 + xy + 2\pi^2 t) \cos(\pi x) \cos(\pi y) & \text{otherwise}.
\end{cases}
\end{split}
\end{equation*}
We produce the simulated noisy data $\tilde u$, as in \eqref{noise}. Next, we  solve ISP by starting CGA with a homogeneous initial guess for the source function $F^0(x,y) = 0.9.$

Figure \ref{fig:Test4ab} shows the reconstruction results in CGA. From Figures \ref{fig:Test4B}, \ref{fig:Test4C}, \ref{fig:Test4F} and \ref{fig:Test4G}, we observe that maximum contrast is achieved, but the location of the reconstructed source is too wide and needs to be reduced. Moreover, for the noise level $\delta = 7$, the reconstructed source given by Figures \ref{fig:Test4D} and \ref{fig:Test4H} is scattered. The relative errors for $\delta = 1$, $\delta = 3$ and $\delta = 7$ are shown in Figure \ref{fig:Test4I}. Table \ref{testm3}, shows the relative error values corresponding to different mesh sizes and noise levels $\delta = 1, 3 \,\, \mbox{and}\,\, 7$. 

Further, one can observe that for the discontinuous data and homogeneous guess at the noise level $\delta =7$, the relative error starts increasing, and our reconstructed source is not good.
\begin{table}[h!] 
\center
\begin{tabular}{ | l | l  |  l | l| l|l|l|l |l|l|l| }
\hline
$N_x =N_y$ & $l$   &   $\Theta_4^{s(1)} $ &   $\Theta_4^{s(3)} $  &   $\Theta_4^{s(7)} $ &   $\Theta_5^{s(1)} $  &   $\Theta_5^{s(3)} $    \\
\hline 
4  & $2$ &  0.1961    &    0.1985   &  0.2120 & 0.1294 &  0.1185  \\
8  & $3$ & 0.1943   &   0.1985  &  0.2014 &    0.0969  &   0.0951    \\ 
16 & $4$ &   0.1940 &   0.1978 & 0.2011 &    0.0939  &  0.0910    \\
32 &  $5$ &   0.1865  &   0.1900  &  0.1987 &   0.0920  &  0.0975      \\
64 &  $6$  &  0.1861  &  0.1894  & 0.1986 &   0.0852  &  0.0964    \\
\hline
\end{tabular}
\caption{Relative errors $\Theta_i^{s(\delta)}$ for the discontinuous source function defined on a circular domain when the simulated data is obtained by the computed solution and $\delta = 1,3,7$ denote the noise level. Moreover, $ \Theta_i^{s(\delta)}  = \frac{\|F - F_{i,m} \|}{\| F\|} $  for mesh sizes $h_l= 2^{-l}, l=2,...,6$ in Experiment $i, i= 4,5$, at the final optimization iteration $m=40$.}
\label{testm3}
\end{table}\\
\noindent \textbf{Experiment 5:} \emph{(Reconstruction of $F$ for a close enough initial guess).}
In this test, we take  $a$, $F$ and $G$ the same as in Experiment 4, and simulated noisy data $\tilde u$ is obtained by the same procedure as in Experiement 4. Next, we  solve ISP by starting
 CGA with an initial guess is close enough to the exact source $F(x,y)$ given by
\begin{eqnarray*}\label{21gaussian3}
F^0(x,y) = \begin{cases}
   1 + x^2y^2 & \text{if } \,\,(x-0.25)^2 + (y-0.3)^2 \leq 0.25; \\
    0.5 + x^2y^2 & \text{ otherwise}.
\end{cases}
\end{eqnarray*}
The Figure \ref{fig:Test5ab} presents the CGA reconstruction results. Figures \ref{fig:Test5A} and \ref{fig:Test5B} demonstrate accurate reconstruction of the source function's location and the achieved maximum contrast. The relative error is shown in Figure \ref{fig:Test5C} and the errors for $\delta=1$ and $\delta=3$ are listed in Table \ref{testm3} under the notations $\Theta_5^{s(1)}$ and $\Theta_5^{s(3)}$, respectively. It is evident that the relative error in this case is comparatively smaller than the homogeneous initial guess considered in Experiment 4.
\begin{figure}
\begin{center}  
\begin{tabular}{cc}
 \subfloat[]{\includegraphics[scale=0.46, clip = true, trim = 0.0cm 0.0cm 0.0cm 0.0cm ]{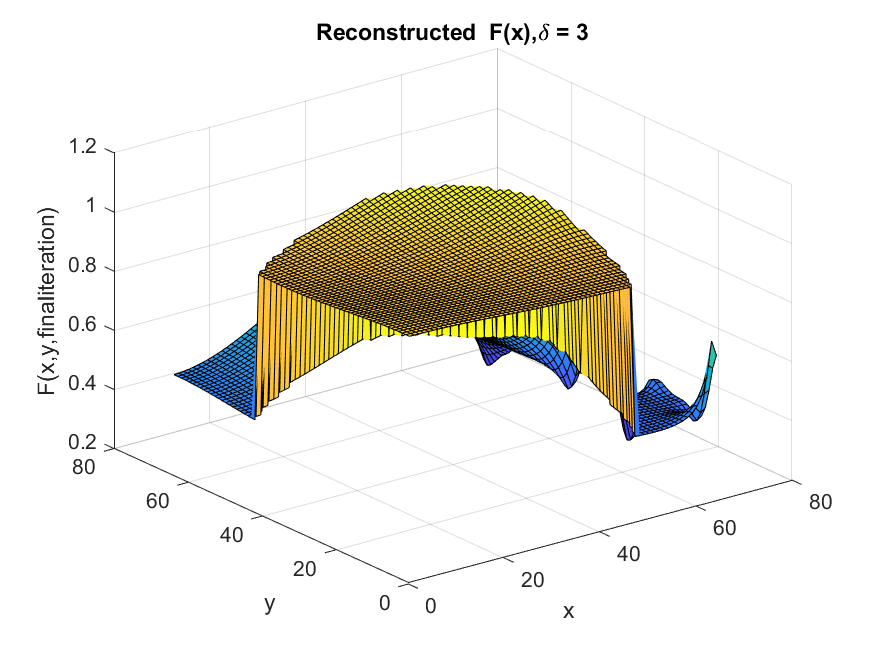}\label{fig:Test5A}} &
 \subfloat[]{\includegraphics[scale=0.46, clip = true, trim = 0.0cm 0.0cm 0.0cm 0.0cm]{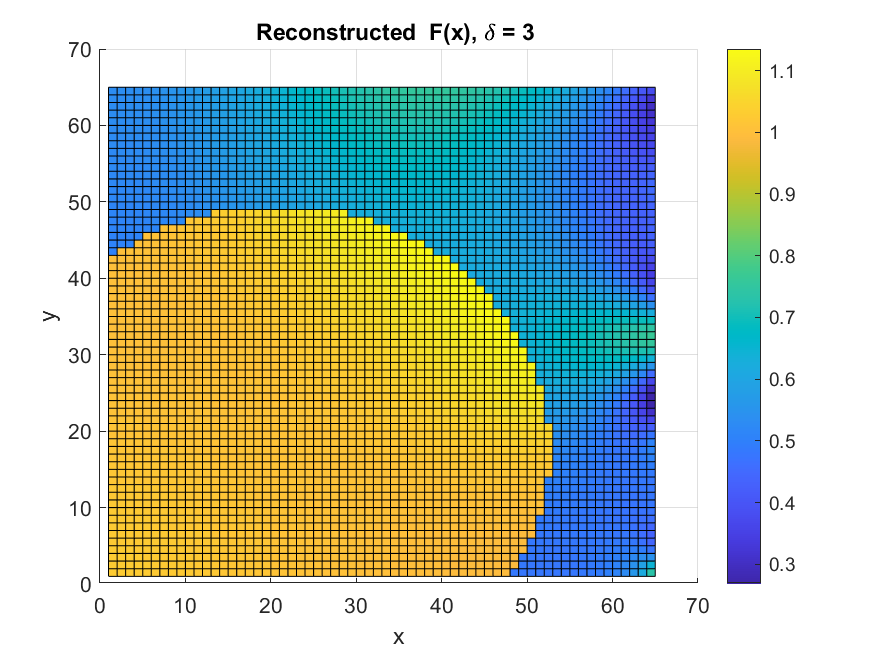}\label{fig:Test5B}} \\
\subfloat[]{\includegraphics[scale=0.46, clip = true, trim = 0.0cm 0.0cm 0.0cm 0.0cm]{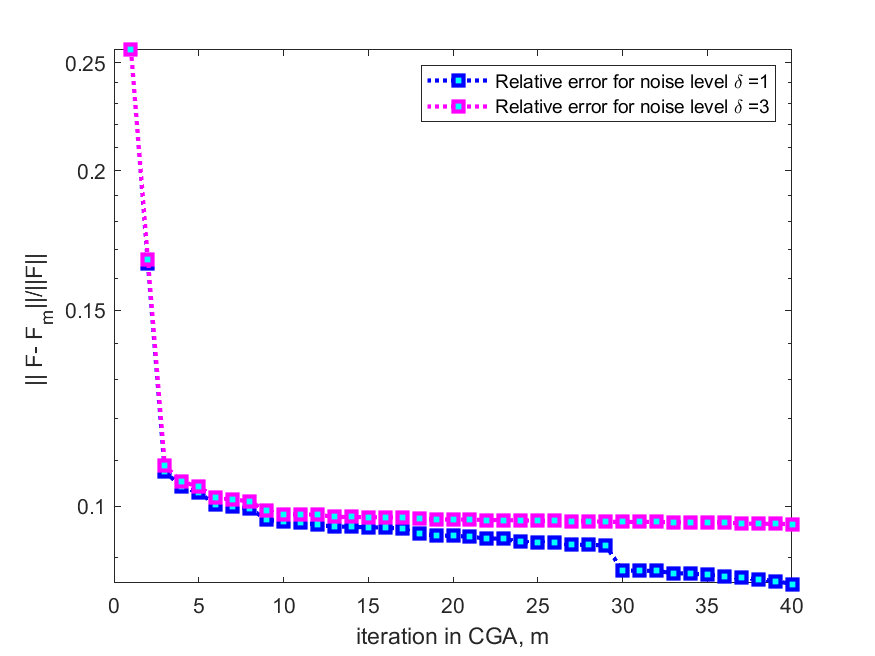}\label{fig:Test5C}} &
\subfloat[]{\includegraphics[scale=0.46, clip = true, trim = 0.0cm 0.0cm 0.0cm 0.0cm]{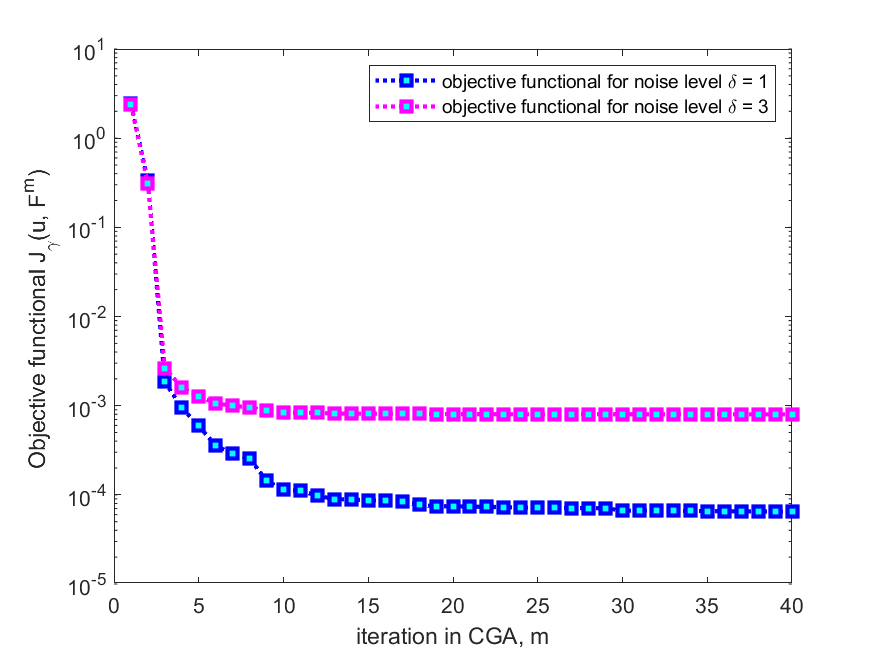}\label{fig:Test5D}} \\ 
\end{tabular}
\end{center}
\caption{Exact $F$ and reconstructed function $F^m$ at the iteration $m= 40$ of CGA. Computations are performed on the mesh with $h = 2^{-6}$.}
 \label{fig:Test5ab}
\end{figure}
\subsection{Reconstruction of the Source Function $F$ in 3D} In this section, we analyze a numerical  example with a close enough initial guess for the source term in the 3D case. \\

\noindent \textbf{Experiment 6:} \emph{(Reconstruction of $F$ for close enough initial guess in 3D).}
In this test, we take known functions  $a=1 + xyz$  and \begin{eqnarray*}\label{G9}
 G(x,y,z,t) =  (1 + xyz + 3\pi^2 t) \cos(\pi x) \cos(\pi y) \cos(\pi z) e^{\frac{(x-0.5)^2 + (y-0.5)^2 + (z-0.5)^2}{m_1}}.
\end{eqnarray*} in \eqref{prb1}.
 The specific source function we aim to reconstruct in this example is represented by the Gaussian function 
\begin{eqnarray*}\label{exsource18}
 F(x,y,z) = e^{-\frac{(x-0.5)^2 + (y-0.5)^2 + (z-0.5)^2}{m_1}}
\end{eqnarray*}
where the constant $m_1=0.1.$ We produce simulated noisy data $\tilde{u}$ using the same procedure as in Experiment 1, 
\begin{eqnarray*}\label{noise4}
\tilde u = u  + normrnd(0,\delta/100,Nx,Ny,Nz),
\end{eqnarray*}
where $Nx, Ny$ and $Nz$ are the number of points in $x$, $y$ and $z$ directions of the domain $\Omega$.
 We begin solving the ISP by initializing CGA with an initial guess for the source function $F^0(x,y,z)= e^{-\frac{(x-0.5)^2 + (y-0.5)^2 + (z-0.5)^2}{m_1}} + \frac{x^2y^2z^2}{10}$. Figure \ref{fig:Test7B} shows the correct reconstruction of the source function's location with maximum contrast. Additionally, in Figures \ref{fig:Test7A} and \ref{fig:Test7B}, the iso-surface value is 0.1. Errors in the reconstruction are shown in Figure \ref{fig:Test7C}. The relative errors for $\delta=1$ and $\delta=3$ are given in Table \ref{testm4} as $\Theta_6^{s(1)}$ and $\Theta_6^{s(3)}$, respectively. 
\begin{table}[h!] 
\center
\begin{tabular}{ | l | l  |  l | l| l|l|l|l |l|l| }
\hline
$N_x =N_y$ & $l$   &    $\Theta_6^{s(1)} $  &   $\Theta_6^{s(3)} $    \\
\hline 
4  & $2$ &     0.0557  &   0.0556  \\
8  & $3$ &    0.0435  &    0.0436     \\ 
16 & $4$ &    0.0312  & 0.0311   \\
\hline
\end{tabular}
\caption{Relative errors $\Theta_i^{s(\delta)}$ for the smooth source function in 3D, when the simulated data is obtained by computed solution and $\delta = 1,3$ denote the noise level. Moreover, $ \Theta_i^{s(\delta)}  = \frac{\|F - F_{i,m} \|}{\| F\|} $  for mesh sizes $h_l= 2^{-l}, l=2,...,4$ in Experiment 6 at the final optimization iteration $m=40$.}
\label{testm4}
\end{table}
\begin{figure}
\begin{center}  
\begin{tabular}{cc}
\subfloat[] {\includegraphics[scale=0.37, clip = true, trim = 0.0cm 0.0cm 0.0cm 0.0cm ]{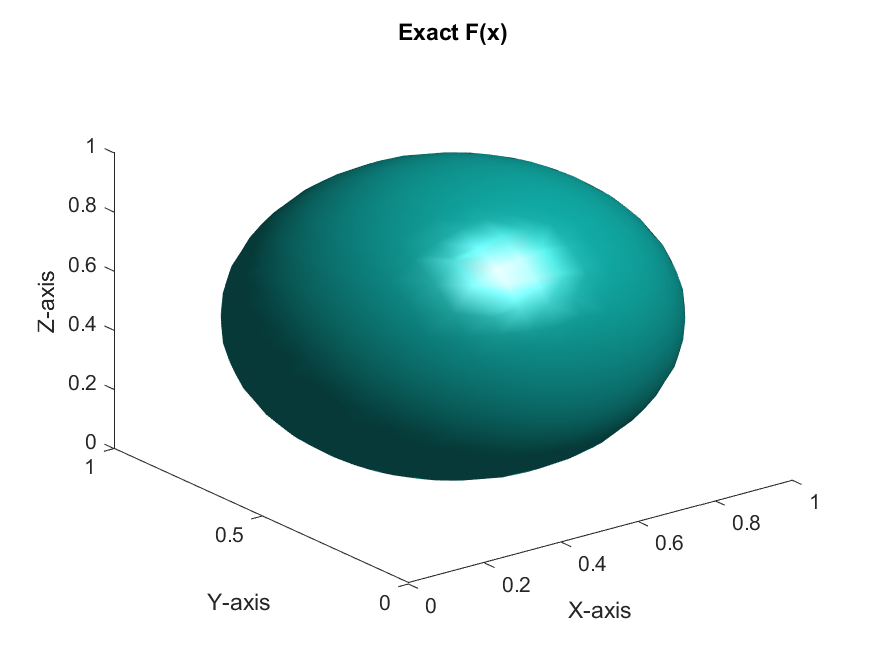} \label{fig:Test7A}} &
 \subfloat[] {\includegraphics[scale=0.37, clip = true, trim = 0.0cm 0.0cm 0.0cm 0.0cm]{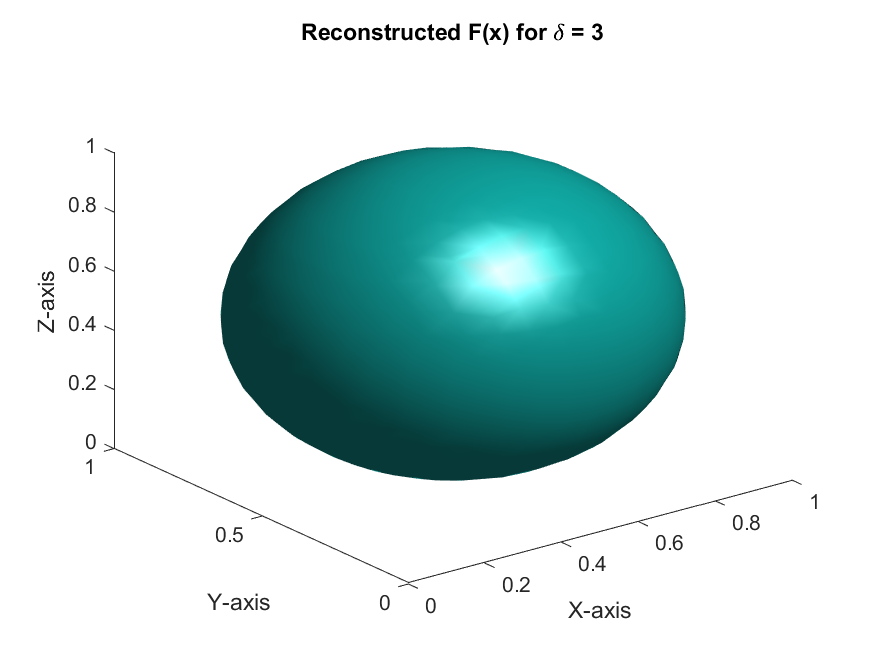} \label{fig:Test7B}}  \\
 \subfloat[] {\includegraphics[scale=0.37, clip = true, trim = 0.0cm 0.0cm 0.0cm 0.0cm]{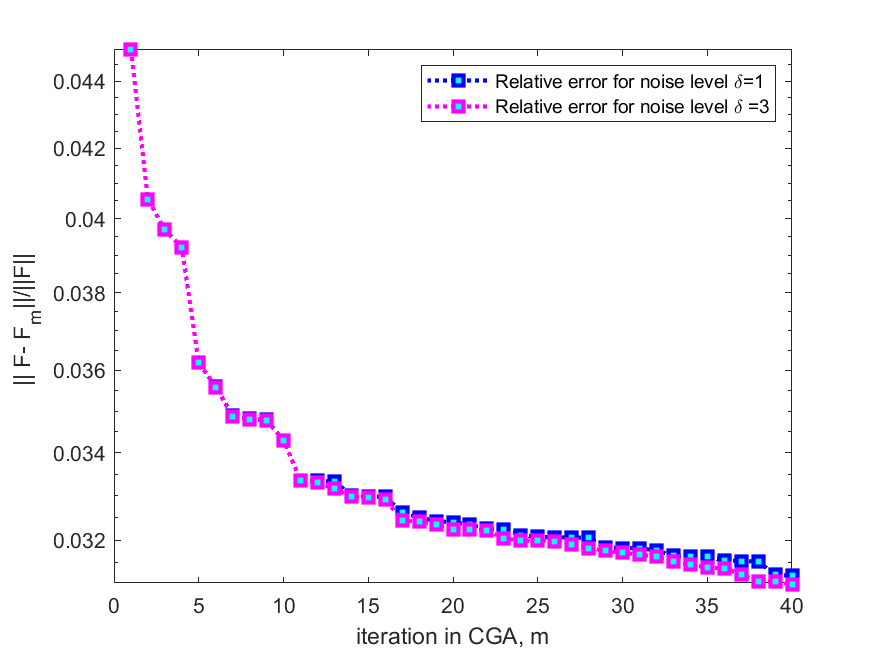} \label{fig:Test7C}} &
\subfloat[]{\includegraphics[scale=0.37, clip = true, trim = 0.0cm 0.0cm 0.0cm 0.0cm]{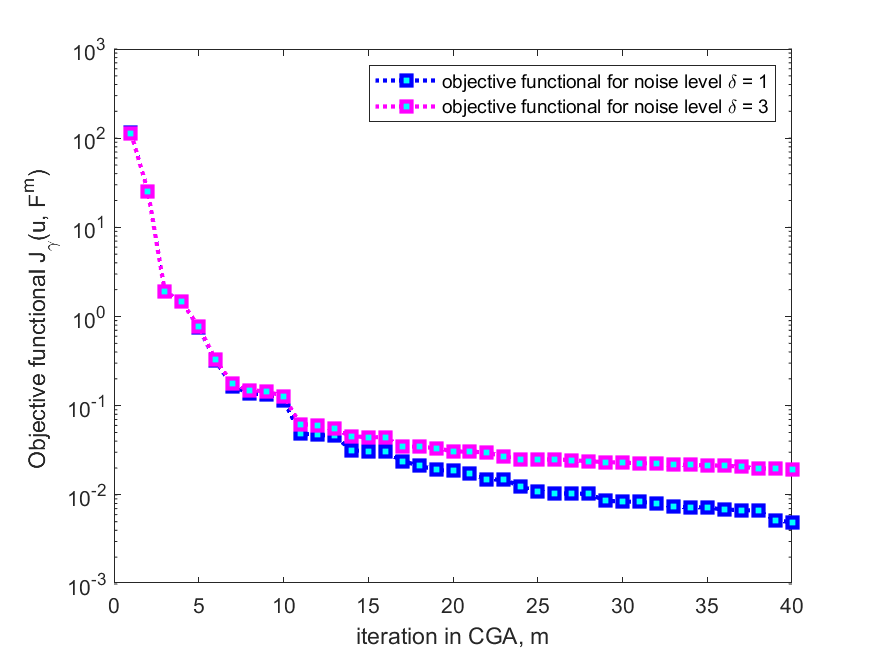} \label{fig:Test7D}} \\
\end{tabular}
\end{center}
\caption{Exact $F$ and reconstructed function $F^m$ at the iteration $m= 40$ of CGA. Computations are performed on the mesh with $h = 2^{-4}$.}
 \label{fig:Test7ab}
 \end{figure}
\section{Conclusions}
\label{sec:concl}
In this work, we present an optimization approach and numerical examples
for the reconstruction of the space-dependent source function $F(x)$ in
the parabolic inverse problem using observations $\tilde u$ at the boundary $\partial_1 \Omega$ of the computational domain
$\Omega\times(0,T)$.
The inverse source problem is approached by minimizing the regularized Tikhonov functional through a Lagrangian method. We present this Lagrangian method, derive the optimality conditions for solving the optimization problem related to the Lagrangian, and formulate a conjugate gradient reconstruction algorithm. The finite difference discretization for both forward and adjoint problem solutions is also provided.
We establish the proof for the Fr\'echet differentiability of the regularized Tikhonov functional, along with the existence result for the solution of the inverse source problem. Additionally, a local stability estimate for the unknown source term is presented.
In the numerical section, we demonstrate a computational study for reconstructing the source function
 $F(x)$. Our numerical tests show that the reconstruction of the source function $F(x)$ remains stable and accurate even with noise levels of $1\%$ and $3\%$.
\section*{Acknowledgments}
The work of the first and third author is supported by the National Board for Higher Mathematics, Govt. of India through the research grant No:02011/13/2022/R{\&}D-II/10206. The research of the second author has been
 supported by the Swedish Research Council grant VR 2024-04459. Additionally, the work is supported by the Swedish Foundation for International Cooperation in
 Research and Education grant STINT MG 2023-9300.

\end{document}